\documentclass[oneside,english]{amsart}
\usepackage[T1]{fontenc}
\usepackage[latin9]{inputenc}
\usepackage{amsthm}
\usepackage{amstext}
\usepackage{amssymb}
\usepackage{esint}
\usepackage[all]{xy}

\makeatletter

\providecommand{\tabularnewline}{\\}

\numberwithin{equation}{section}
\numberwithin{figure}{section}
\theoremstyle{plain}
\newtheorem{thm}{\protect\theoremname}[section]
  \theoremstyle{plain}
  \newtheorem{conjecture}[thm]{\protect\conjecturename}
  \theoremstyle{definition}
  \newtheorem{defn}[thm]{\protect\definitionname}
  \theoremstyle{remark}
  \newtheorem{notation}[thm]{\protect\notationname}
  \theoremstyle{plain}
  \newtheorem{lem}[thm]{\protect\lemmaname}
  \theoremstyle{remark}
  \newtheorem{rem}[thm]{\protect\remarkname}
  \theoremstyle{plain}
  \newtheorem{cor}[thm]{\protect\corollaryname}
  \theoremstyle{definition}
  \newtheorem{condition}[thm]{\protect\conditionname}
  \theoremstyle{plain}
  \newtheorem{prop}[thm]{\protect\propositionname}
  \theoremstyle{remark}
  \newtheorem*{claim*}{\protect\claimname}
  \theoremstyle{definition}
  \newtheorem{example}[thm]{\protect\examplename}

\subjclass[2010]{11R21,11R45,11G50,14G05, 14E16}

\usepackage{yhmath} 

\makeatother

\usepackage{babel}
  \providecommand{\claimname}{Claim}
  \providecommand{\conditionname}{Condition}
  \providecommand{\conjecturename}{Conjecture}
  \providecommand{\corollaryname}{Corollary}
  \providecommand{\definitionname}{Definition}
  \providecommand{\examplename}{Example}
  \providecommand{\lemmaname}{Lemma}
  \providecommand{\notationname}{Notation}
  \providecommand{\propositionname}{Proposition}
  \providecommand{\remarkname}{Remark}
\providecommand{\theoremname}{Theorem}

\begin{document}

\title{Manin's conjecture vs.\ Malle's conjecture}

\author{Takehiko Yasuda}
\begin{abstract}
By a heuristic argument, we relate two conjectures. One is a version
of Manin's conjecture about the distribution of rational points on
a Fano variety. We concern specific singular Fano varieties, namely
quotients of projective spaces by finite group actions, and their
singularities play a key role. The other conjecture is a generalization
of Malle's conjecture about the distribution of extensions of a number
field. Main tools are several Dirichlet series and previously obtained
techniques, especially the untwisting, for the counterpart over a
local field.
\end{abstract}

\address{Department of Mathematics, Graduate School of Science, Osaka University,
Toyonaka, Osaka 560-0043, Japan, tel:+81-6-6850-5326, fax:+81-6-6850-5327}

\curraddr{(Until Aug.\ 2015) Max Planck Institute for Mathematics, Vivatsgasse
7, 53111 Bonn, Germany}

\email{takehikoyasuda@math.sci.osaka-u.ac.jp, highernash@gmail.com}

\keywords{Manin's conjecture, Malle's conjecture, the McKay correspondence,
singular Fano varieties, heights}

\maketitle
\global\long\def\AA{\mathbb{A}}
\global\long\def\PP{\mathbb{P}}
\global\long\def\NN{\mathbb{N}}
\global\long\def\GG{\mathbb{G}}
\global\long\def\ZZ{\mathbb{Z}}
\global\long\def\QQ{\mathbb{Q}}
\global\long\def\CC{\mathbb{C}}
\global\long\def\FF{\mathbb{F}}
\global\long\def\LL{\mathbb{L}}
\global\long\def\RR{\mathbb{R}}
\global\long\def\MM{\mathbb{M}}
\global\long\def\SS{\mathbb{S}}

\global\long\def\bx{\boldsymbol{x}}
\global\long\def\by{\boldsymbol{y}}
\global\long\def\bf{\mathbf{f}}
\global\long\def\ba{\mathbf{a}}
\global\long\def\bs{\mathbf{s}}
\global\long\def\bt{\mathbf{t}}
\global\long\def\bw{\mathbf{w}}
\global\long\def\bb{\mathbf{b}}
\global\long\def\bv{\mathbf{v}}
\global\long\def\bp{\mathbf{p}}
\global\long\def\bq{\mathbf{q}}
\global\long\def\bm{\mathbf{m}}
\global\long\def\bj{\mathbf{j}}
\global\long\def\bM{\mathbf{M}}
\global\long\def\bd{\mathbf{d}}

\global\long\def\cN{\mathcal{N}}
\global\long\def\cW{\mathcal{W}}
\global\long\def\cY{\mathcal{Y}}
\global\long\def\cM{\mathcal{M}}
\global\long\def\cF{\mathcal{F}}
\global\long\def\cX{\mathcal{X}}
\global\long\def\cE{\mathcal{E}}
\global\long\def\cJ{\mathcal{J}}
\global\long\def\cO{\mathcal{O}}
\global\long\def\cD{\mathcal{D}}
\global\long\def\cZ{\mathcal{Z}}
\global\long\def\cR{\mathcal{R}}
\global\long\def\cC{\mathcal{C}}
\global\long\def\cL{\mathcal{L}}
\global\long\def\cV{\mathcal{V}}

\global\long\def\fs{\mathfrak{s}}
\global\long\def\fp{\mathfrak{p}}
\global\long\def\fm{\mathfrak{m}}
\global\long\def\fX{\mathfrak{X}}
\global\long\def\fV{\mathfrak{V}}
\global\long\def\fx{\mathfrak{x}}
\global\long\def\fv{\mathfrak{v}}
\global\long\def\fY{\mathfrak{Y}}

\global\long\def\rv{\mathbf{\mathrm{v}}}
\global\long\def\rx{\mathrm{x}}
\global\long\def\rw{\mathrm{w}}
\global\long\def\ry{\mathrm{y}}
\global\long\def\rz{\mathrm{z}}
\global\long\def\bv{\mathbf{v}}
\global\long\def\bw{\mathbf{w}}
\global\long\def\sv{\mathsf{v}}
\global\long\def\sx{\mathsf{x}}
\global\long\def\sw{\mathsf{w}}

\global\long\def\Spec{\mathrm{Spec}\,}
\global\long\def\Hom{\mathrm{Hom}}

\global\long\def\Var{\mathrm{Var}}
\global\long\def\Gal{\mathrm{Gal}}
\global\long\def\Jac{\mathrm{Jac}}
\global\long\def\Ker{\mathrm{Ker}}
\global\long\def\Image{\mathrm{Im}}
\global\long\def\Aut{\mathrm{Aut}}
\global\long\def\st{\mathrm{st}}
\global\long\def\diag{\mathrm{diag}}
\global\long\def\characteristic{\mathrm{char}}
\global\long\def\tors{\mathrm{tors}}
\global\long\def\sing{\mathrm{sing}}
\global\long\def\red{\mathrm{red}}
\global\long\def\ord{\mathrm{ord}}
\global\long\def\pt{\mathrm{pt}}
\global\long\def\op{\mathrm{op}}
\global\long\def\Val{\mathrm{Val}}
\global\long\def\Res{\mathrm{Res}}
\global\long\def\Pic{\mathrm{Pic}}
\global\long\def\disc{\mathrm{disc}}
\global\long\def\height{\mathrm{ht}}
 \global\long\def\length{\mathrm{length}}
\global\long\def\sm{\mathrm{sm}}
\global\long\def\rank{\mathrm{rank}}
\global\long\def\age{\mathrm{age}}
\global\long\def\et{\mathrm{et}}
\global\long\def\hom{\mathrm{hom}}
\global\long\def\tor{\mathrm{tor}}
\global\long\def\reg{\mathrm{reg}}
\global\long\def\cont{\mathrm{cont}}
\global\long\def\crep{\mathrm{crep}}
\global\long\def\Stab{\mathrm{Stab}}
\global\long\def\discrep{\mathrm{discrep}}
\global\long\def\mld{\mathrm{mld}}

\global\long\def\GL{\mathrm{GL}}
\global\long\def\codim{\mathrm{codim}}
\global\long\def\Val{\mathrm{Val}}
\global\long\def\ur{\mathrm{ur}}
\global\long\def\Eff{\mathrm{PEf}}
\global\long\def\Tor{\textrm{-}\mathrm{Rin}}
\global\long\def\Fie{\textrm{-}\mathrm{Fie}}
\global\long\def\prim{\mathrm{prim}}
\global\long\def\cHom{\mathcal{H}om}
\global\long\def\cSpec{\mathcal{S}pec}
\global\long\def\Proj{\mathrm{Proj}\,}
\global\long\def\modified{\mathrm{modif}}
\global\long\def\ind{\mathrm{ind}}
\global\long\def\Conj{\mathrm{Conj}}
\global\long\def\KConj{K\textrm{-}\mathrm{Conj}}
\global\long\def\fie{\textrm{-}\mathrm{fie}}
\global\long\def\NS{\mathrm{NS}}
\global\long\def\Disc{\mathrm{Disc}}
\global\long\def\Peyre{\mathrm{Peyre}}

\tableofcontents{}

\section{Introduction\label{sec:Introduction}}

The aim of this paper is to relate the distribution of rational points
on a Fano variety to the distribution of extensions of a number field.
Central in the two problems are respectively Manin's conjecture \cite{Franke:1989go}
and Malle's conjecture \cite{MR1884706,MR2068887}. We consider variants
of these conjectures, and see by a heuristic argument that they explain
each other. 

Let $K$ be a number field and $X$ a Fano variety over $K$ having
at worst log terminal singularities. Giving an adelic metric to the
anti-canonical divisor $-K_{X}$ defines a height function $H:X(K)\to\RR_{>0}$.
For a subset $U\subset X(K)$ and for a real number $B>0$, we let
\[
\cN_{U}(B):=\sharp\{x\in U\mid H(x)\le B\}.
\]
We are interested in the asymptotic behavior of this number as $B$
tends to infinity. As a subset $U$, we consider a \emph{cothin subset,
}the complement of a thin subset. In turn, a \emph{thin subset }of
$X(K)$ is defined as a subset contained in the image of $Y(K)$ for
a generically finite morphism $Y\to X$ admitting no rational section. 

We consider the following version of Manin's conjecture:
\begin{conjecture}[Conjecture \ref{conj:BT}]
\label{conj:intro-manin}Suppose that $X$ is a log terminal Fano
variety. 
\begin{enumerate}
\item If $X$ is canonical (has only canonical singularities) and $K$ is
sufficiently large, then for a sufficiently small cothin subset $U\subset X(K)$,
we have 
\[
\cN_{U}(B)\sim CB(\log B)^{\rho(X)+\gamma(X)-1},
\]
where $C$ is a positive constant, $\rho(X)$ is the Picard number
of $X$ and $\gamma(X)$ is the number of crepant divisors over $X$. 
\item If $X$ is not canonical and if $K$ is sufficiently large, then for
any cothin subset $U\subset X(K)$, we have
\[
\cN_{U}(B)\sim CB^{\alpha}(\log B)^{\beta}
\]
with $\alpha>1$ and $\beta\ge0$. 
\end{enumerate}
\end{conjecture}
The use of cothin subsets was first suggested by Peyre \cite{MR2019019}
and more seriously discussed by Le Rudulier \cite{LeRudulierSurface}
and Browning--Loughran \cite{Browning:2013aa}. This modification
of considered subsets is appropriate for quotient varieties which
we will consider. It is also explained in relation to the distribution
of number field extensions (Remark \ref{rem:fields-than-rings}).
The formula in the first assertion is a special case of a formula
considered by Batyrev--Tschinkel \cite{MR1679843} and also deduced
from a conjecture by Batyrev--Manin \cite{MR1032922}. 

We will relate this version of Manin's conjecture with the distribution
of Galois extensions of a number field. Let $K$ still denote a number
field and let $G$ be a finite group. We mean by a $G$-\emph{field
(over $K$) }a finite Galois\emph{ }extension $L/K$ endowed with
an isomorphism $\Gal(L/K)\cong G$. We denote the $G$-equivariant
isomorphism classes of $G$-fields by $G\Fie(K)$. Let $V$ be a finite-dimensional
faithful $G$-represenatation $V$ over $K$ and $S$ a finite set
of places of $K$ containing all infinite places such that $V$ is
defined over the $S$-integer ring $\cO_{S}$. We will define the
notion of $V$\emph{-discrimiant} of $G$-fields and denote them by
$D_{L}^{V}$. This invariant is the global version of counting functions
for extensions of a local field which appeared in the study of the
wild McKay correspondence \cite{Wood-Yasuda-I,MR3230848,wild-p-adic,Yasuda:2013fk,Yasuda:2014fk2}.
Dummit \cite{Dummit:2014vb} introduced a probably closely related
invariant of $\rho$-\emph{discriminants }and studied the distribution
of $G$-fields with respect to it.\emph{ }Although not knowing the
precise relation between $V$- and $\rho$-discriminants, the author
was influenced by Dummit's work. For a real number $B>0$, let $N_{G,V,K}(B)$
be the number of isomorphism classes of $G$-fields $L$ over $K$
with $D_{L}^{V}\le B$. To formulate our conjecture on the distribution
of $G$-fields, we recall that there exists the \emph{age }function
$\age:G\to\QQ_{\ge0}$ associated to the given $G$-represenatation
$V$, which often appears in the context of the McKay correspondence
(see \cite{MR1463181}). Let $\age(G)$ be the minimum of $\age(g)$,
$g\in G\setminus\{1\}$ and $\upsilon(G)$ the number of non-trivial
$K$-conjugacy classes (conjugacy classes of $G$ modulo a $\Gal(\overline{K}/K)$-action)
with minimum age. We raise the following conjecture:
\begin{conjecture}[Conjecture \ref{conj:general Malle}]
\label{conj:intro-malle}If $K$ is sufficiently large, then
\[
N_{G,V,K}(B)\sim CB^{1/\age(G)}(\log B)^{\upsilon(G)}.
\]

\end{conjecture}
Except that $K$ is sufficiently large, this contains Malle's conjecture
\cite{MR2068887} as a special case: if $G$ is a transitive subgroup
of the symmetric group $S_{n}$ and if $V=K^{2n}$ is the direct sum
of two copies of the natural permutation representation, then the
above conjecture is equivalent to Malle's conjecture. For a technical
reason, we also introduce a variant of $V$-discriminant, the \emph{extended
$V$-discriminant, }denoted $\tilde{D}_{L}^{V}$. Since for some constants
$C_{1},C_{2}>0$, $C_{1}D_{L}^{V}\le\tilde{D}_{L}^{V}\le C_{2}D_{L}^{V}$,
we expect that the similarly defined number $\tilde{N}_{G,V,K}(B)$
for $\tilde{D}_{L}^{V}$ would satisfy an asymptotic formula of the
same form. 

To relate Conjectures \ref{conj:intro-manin} and \ref{conj:intro-malle}
to each other, we identify the given $G$-representation $V$ with
the affine space $\AA_{K}^{d}$ and consider the quotient variety
$X=\AA_{K}^{d}/G$ and its compactification $\overline{X}=\PP_{K}^{d}/G$.
If the map $\PP_{K}^{d}\to\overline{X}$ is étale in codimension one,
then $\overline{X}$ is a log terminal Fano variety whose minimal
log discrepancy is equal to $\age(G)$. We define \emph{primitive
$K$-points }of $X$ as those $K$-points of $X$ not coming from
$K$-points of intermediate covers of $\AA_{K}^{d}\to X$. We propose
the set $X_{\prim}(K)$ of primitive $K$-points as the sufficiently
small cothin subset $U\subset\overline{X}(K)$ in Conjecture \ref{conj:intro-manin}.
We fix an adelic metric on $-K_{\overline{X}}$, which determines
height functions on $X$ and $V$. We consider the following Dirichlet
series: the \emph{height zeta functions}
\begin{gather*}
Z_{X_{\prim}(K)}(s)=\sum_{x\in X_{\prim}(K)}H(x)^{-s},\\
Z_{V(K)}(s)=\sum_{y\in V(K)}H(y)^{-s},
\end{gather*}
and the \emph{extended $V$-Discriminant zeta function},
\[
Z^{\disc}(s)=\sum_{L\in G\Fie(K)}\left(\tilde{D}_{L}^{V}\right)^{-s}.
\]
From a heuristic argument, we expect that the product $Z^{\disc}(s)Z_{V(K)}(s)$
would be a good approximation of $Z_{X_{\prim}(K)}(s)$, and they
would have the right-most poles of the same place and the same order,
assuming suitable meromorphic continuation of these zeta functions
beyond the abscissae of convergence. Under the heuristic assumption
we will make, Conjectures \ref{conj:intro-manin} implies a weak version
of \ref{conj:intro-malle} and vice versa. Moreover, when $\age(G)=1$,
then the two conjectures become equivalent. If we look even at the
residues of these Dirichlet series at the right-most poles, then we
might be able to relate Peyre's refinement \cite{MR1340296} of Manin's
conjecture in terms of Tamagawa measures (and Batyrev--Tschinkel's
generalization \cite{MR1679843} of it) with Bhargava's probabilistic
heuristics \cite{MR2354798} about the distribution of number fields.
However this subject is not pursued in this paper. 

The mentioned heuristic argument relies on the \emph{untwisting technique}.
It was first used by Denef--Loeser \cite{MR1905024} in their proof
of a version of the McKay correspondence with an explicit construction.
Later it was generalized and made more intrinsic by the author \cite{Yasuda:2013fk,Yasuda:2014fk2},
which can be easily translated into the case of global fields. For
each $L\in G\Fie(K)$, we construct the \emph{untwisting variety}
$V^{|L|}$, which is again isomorphic to $\AA_{K}^{d}$, but there
exists no canonical isomorphism between $V$ and $V^{|L|}$. There
exists a one-to-one correspondence of $G$-equivariant $L$-points
of $V$ and (non-equivariant) $K$-points of $V^{|L|}$. It is then
straightforward to see that we have a bijection 
\[
\bigsqcup_{L\in G\Fie(K)}V_{\ur}^{|L|}(K)/\Aut(L)\to X_{\prim}(K).
\]
Here $V_{\ur}^{|L|}$ is the preimage of $X_{\ur}$ in $V^{|L|}$
with $X_{\ur}\subset X$ the unramified locus of $V\to X$. Therefore,
if we define heights on $V^{|L|}$ as the pull-backs of the one on
$X$, then we trivially have the equality among height zeta functions,
\[
Z_{X_{\prim}(K)}(s)=\frac{1}{\sharp Z(G)}\sum_{L\in G\Fie(K)}Z_{V_{\ur}^{|L|}(K)}(s),
\]
noting $\sharp Z(G)=\sharp\Aut(L)$. Using Peyre's refinement \cite{MR1340296}
of Manin's conjecture and comparing Tamagawa measures on $V$ and
$V^{|L|}$, we heuristically expect that $Z_{V_{\ur}^{|L|}(K)}(s)$
would be approximated by $\left(\tilde{D}_{L}^{V}\right)^{-s}Z_{V(K)}(s)$,
and hence that $Z_{X_{\prim}(K)}(s)$ would be approximated by $Z^{\disc}(s)Z_{V(K)}(s)$. 

To end this introduction, we here mention related works. In the paper
\cite[page 153]{Ellenberg:2005bn} of Ellenberg and Venkatesh, it
was mentioned, as a comment by Tschinkel, a similarity between their
work on Malle's conjecture and Batyrev's one on rational points on
Fano varieties. In another paper of theirs \cite[page 732]{Ellenberg:2006js},
the relation between Malle's conjecture and Manin's conjecture was
more explicitly noted. In the same paper, they use the field of multi-symmetric
functions, which is the function field of the above quotient variety
$X$. The approach using $X$ or its function field is regarded as
a revisitation of Noether's approach to the inverse Galois problem
(see \cite{MR2363329}), with more emphasis on the quantitative aspect.
As more recent and more similar works, we mention Le Rudulier's papers
\cite{LeRudulierLine,LeRudulierSurface} and the author's unpublished
manuscript \cite{densities}. The former considers $S_{n}$-permutation
actions on products of projective lines or planes, and the latter
considers permutation actions of more general finite groups on products
of projective spaces. Our work can be also considered as the global
version of the wild McKay correspondence\footnote{The previously obtained \emph{tame} McKay correspondence \cite{MR1677693,MR1905024}
is of course relevant too. Indeed we mainly concern the tame situation.
However, in the tame and \emph{local} case, being somehow trivial,
the viewpoint of counting extensions was seemingly missing or only
implicit. It should be also related, the McKay correspondence in the
context of curve counting theories like Gromov--Witten and Donaldson--Thomas,
on which the author lacks expertise. } (see \cite{Wood-Yasuda-I,MR3230848,wild-p-adic,Yasuda:2013fk,Yasuda:2014fk2}).
Mäki \cite{Maki:1993et} studied the distribution of $G$-fields for
abelian $G$ with respect to conductors rather than classically used
discriminants and Wood \cite{Wood:2010gs} considered more general
counting functions, in particular, Artin conductors (see Section 5
of her paper). In the case where $V$ is a \emph{balanced} $G$-representation,
$V$-discrimiants and Artin conductors basically coincide, which was
proved in \cite{Wood-Yasuda-I} in the case of local fields. 

The outline of the paper is as follows. In Section \ref{sec:Notation-and-convention}
we set up frequently used notation and convention. After reviewing
basic materials on adelic metrics and heights in Section \ref{sec:Metrics-and-heights}
and singularities in Section \ref{sec:Singularities}, we recall several
versions of Manin's conjecture in Section \ref{sec:Rational-points-on}.
In Section \ref{sec:Height-zeta-functions}, we briefly recall height
zeta functions and a version of Tauberian theorem. In Section \ref{sec:Peyre's-constants-for}
we recall Peyre's refinement of Manin's conjecture in the case of
projective spaces. Sections \ref{sec:Quotients-of-projective} to
\ref{sec:Modified-heights} are devoted to the study of quotients
of projective spaces by finite group actions in the context of Manin's
conjecture. In Section \ref{sec:Distribution-of-number} we recall
Malle's conjecture and propose a generalization of it. In the final
Section we relate Manin's conjecture and Malle's conjecture as a consequence
of materials prepared in earlier sections.

\subsection*{Acknowledgments}

The author would like to thank the following people for helpful discussion
and comments; Jordan Ellenberg, Daniel Loughran, Shinnosuke Okawa,
Ramin Takloo-Bighash, Takashi Taniguchi, Melanie Wood, Takuya Yamauchi,
Seidai Yasuda and Akihiko Yukie. 

A part of this work was done during the author's stay at the Max Planck
Institute for Mathematics. He is grateful for its hospitality.

\section{Notation and convention\label{sec:Notation-and-convention}}

Throughout the paper, we work over a base number field denoted by
$K$. We write the set of places of $K$ as $M_{K}$ and the subsets
of finite and infinite places as $M_{K,f}$ and $M_{K,\infty}$ respectively.
For $\fp\in M_{K}$, $K_{\fp}$ denotes the corresponding completion
of $K$. The normalized absolute value on $K_{\fp}$ is denoted by
$\left\Vert \cdot\right\Vert _{\fp}$: if $p$ is the place of $\QQ$
with $\fp\mid p$ and $\left|\cdot\right|_{p}$ denotes the $p$-adic
absolute value, then $\left\Vert \cdot\right\Vert _{\fp}=\left|N_{K_{\fp}/\QQ_{p}}(\cdot)\right|_{p}$.
When $\fp$ is finite, we write its integer ring by $\cO_{\fp}$,
the residue field by $\kappa_{\fp}$ and the cardinality of $\kappa_{\fp}$
by $N_{\fp}$. We denote by $S$ a finite subset of $M_{K}$ containing
$M_{K,\infty}$, and by $S^{c}$ its complement $M_{K}\setminus S$.
We denote the integer ring of $K$ by $\cO_{K}$ and the $S$-integer
ring by $\cO_{S}$. 

We sometimes denote an arbitrary field of characteristic zero by $F$.
We denote by $\overline{F}$ an algebraic closure of $F$. A \emph{variety
}over $F$ means a separated integral scheme of finite type over $F$.
A \emph{$\QQ$-divisor on }a variety or a scheme means a $\QQ$-Cartier
$\QQ$-Weil divisor. A \emph{divisor over }a variety $X$ means an
equivalence classes of prime divisors on proper birational modifications
of $X$, or equivalently a divisorial valuation on the function field
of $X$ having its center on $X$. \emph{Points }of schemes usually
mean morphisms: for instance, for an $F$-variety $X$ and an $F$-algebra
$L$, an $L$-point of $X$ means an $F$-morphism $\Spec L\to X$.
The set of $L$-points of $X$ is then denoted by $X(L)$. We denote
by $X_{L}$ the base change $X\otimes_{F}L$.

The symbol $L$ usually denotes an algebra over a field, except that
in Section \ref{sec:Metrics-and-heights}, it denotes an invertible
sheaf.

Formulas of the form $\sim CB^{\alpha}(\log B)^{\beta}$ are always
asymptotic and we understand that $B$ tends to the infinity and $C$
is a certain positive real constant independent of $B$.

\section{Adelic metrics and heights\label{sec:Metrics-and-heights}}

In this section, we briefly recall adelic metrics on invertible sheaves
and associated height functions. For details we refer the reader to
\cite{MR1340296,MR2019019,MR2647601}. We slightly generalize these
notions to $\QQ$-divisors, which is rather straightforward. 

Let $X$ be a projective variety over a number field $K$ and $L$
an invertible sheaf on $X$. For any field extension $F/K$ and $x\in X(F)$,
we denote by $L(x)$ the pull-back of $L$ by $x:\Spec F\to X$. For
$\fp\in M_{K}$, a $\fp$-\emph{adic metric} on $L$ is the data of
$\fp$-adic norms on $L(x)$, $x\in X(K_{\fp})$, satisfying the following
continuity: for a section $s$ of $L$ over a Zariski open subset
$U\subset X$, the function 
\[
U(K_{\fp})\to\RR_{>0},\,x\mapsto\left\Vert s(x)\right\Vert _{\fp}
\]
is continuous.

For a finite set $S\subset M_{K}$ containing all infinite places,
a \emph{model }of $(X,L)$ over $\cO_{S}$ means the pair $(\cX,\cL)$
of an integral projective $\cO_{S}$-scheme $\cX$ and an invertible
sheaf $\cL$ on it such that $X$ is the generic fiber of $\cX\to\Spec\cO_{S}$
and $L$ is the pullback of $\cL$ to $X$. We also say that $\cL$
is a model of $L$. A model determines a specific choice of a $\fp$-adic
metric on $L$ for $\fp\in S^{c}$ as follows. For $x\in X(K_{\fp})$,
which is identified with an $\cO_{\fp}$-point of $\cX$, the $K_{\fp}$-line
$L(x)$ contains the pull-back $\cL(x)$ of $\cL$ by $x:\Spec\cO_{\fp}\to\cX$.
Fixing an isomorphism $\alpha:L(x)\xrightarrow{\sim}K_{\fp}$ mapping
$\cL(x)$ onto $\cO_{\fp}$, we define a $\fp$-adic metric $\left\Vert \cdot\right\Vert _{\fp}$
on $L(x)$ by 
\[
\left\Vert y\right\Vert _{\fp}:=\left\Vert \alpha(y)\right\Vert _{\fp}.
\]

An\emph{ (adelic) metric} on $L$ is a collection $(\left\Vert \cdot\right\Vert _{\fp})_{\fp\in M_{K}}$
of $\fp$-adic metrics on $L$ such that there exists a model $(\cX,\cL)$
of $(X,L)$, say over $\cO_{S}$, defining $\left\Vert \cdot\right\Vert _{\fp}$
for $\fp\in S^{c}$. We call an invertible sheaf endowed with a metric
a \emph{metrized invertible sheaf. }For a metrized invertible sheaf
$L$, we often fix such a model $(\cX,\cL)$ as above and call $\cL$
\emph{the }model of $L$.

Given a metrized invertible sheaf $L$, the \emph{height }of $x\in X(K)$
is then defined as

\[
H(x):=\prod_{\fp\in M_{K}}\left\Vert y\right\Vert _{\fp}^{-1}
\]
for an arbitrary $y\in L(x)\setminus\{0\}$. Thanks to the product
formula, this is independent of the choice of $y$. When we need to
specify the sheaf $L$, we write $H_{L}$ for $H$. If $L'$ is another
metrized invertible sheaf, then the tensor product $L\otimes L'$
has the induced metric and the associated height function satisfies
\[
H_{L\otimes L'}(x)=H_{L}(x)\cdot H_{L'}(x).
\]

For our purpose, it is more convenient to use the divisorial notation
and to generalize to $\QQ$-divisors. By a \emph{$\QQ$-divisor} on
a scheme, we mean a $\QQ$-Cartier $\QQ$-Weil divisor. We say that
a $\QQ$-divisor $D$ on $X$ is \emph{metrized }if for a certain
integer $r>0$, $rD$ is Cartier and $\cO_{X}(rD)$ is metrized. For
such a metrized $\QQ$-divisor $D$, we define the height function
$X(K)\to\RR_{>0}$ by 
\[
H_{D}(x):=H_{rD}(x)^{1/r}.
\]
For two metrized $\QQ$-divisors $D$ and $D'$, the sum $D+D'$ is
naturally metrized and we have 
\[
H_{D+D'}(x)=H_{D}(x)\cdot H_{D'}(x).
\]
Let us consider a $\QQ$-divisor $\cD$ on $\cX$ whose restriction
to $X$ is $D$. When the metrized sheaf $\cO_{X}(rD)$ has the model
$\cO_{\cX}(r\cD)$, then we call $\cD$ the \emph{model }of the metrized
divisor $D$. 

If $f:Y\to X$ is a morphism of $K$-varieties and $D$ is a metrized
$\QQ$-divisor on $X$, then the pull-back divisor $f^{*}D$ is naturally
metrized. For $y\in Y(K)$, we have
\begin{equation}
H_{f^{*}D}(y)=H_{D}(f(y)).\label{eq:ht-eq}
\end{equation}

\section{Singularities\label{sec:Singularities}}

In this section, we recall the notion of \emph{discrepancies}, which
are important invariants of singularities especially in the birational
geometry, and related notions. We refer the reader to \cite{MR3057950}
for more details. 

Let $X$ be a normal variety over an arbitrary field $F$ of characteristic
zero with the canonical divisor $K_{X}$ $\QQ$-Cartier. For a proper
birational morphism $f:Y\to X$ with $Y$ normal, we can uniquely
write 
\begin{equation}
K_{Y}=f^{*}K_{X}+\sum_{E}a(E)\cdot E\label{eq:discrep}
\end{equation}
where $a(E)\in\QQ$ and $E$ runs over the exceptional prime divisors.
The rational number $a(E)+1$ is called the \emph{log discrepancy
}of $E$ ($a(E)$ itself is called the \emph{discrepancy}).

A \emph{divisor over $X$ }means a prime divisor on a normal variety
$Y$ proper birational over $X$. Two divisors over $X$ should be
considered as identical if they give the same valuation on the function
field of $X$. Namely a divisor over $X$ is a divisorial valuation
of the function field $K(X)$ having its center on $X$. The log discrepancy
is actually an invariant of a divisor over $X$, that is, independent
of the birational model on which the divisor lies. 
\begin{defn}
The \emph{minimal log discrepancy }of $X$ is defined as
\[
\mld(X):=\inf_{E}\{a(E)+1\},
\]
where $E$ runs over all divisors over $X$. We call $X$ (or singularities
of $X$) \emph{terminal (resp. canonical, log terminal) }if $\mld(X)>1$
(resp. $\ge1$, $>0$).
\end{defn}
Suppose that $X$ is log terminal. For a log resolution $f:Y\to X$,
we have 
\[
\mld(X)=\min\{a(E)+1\mid E\text{ exceptional divisor of }f\}.
\]

\begin{defn}
We define a\emph{ crepant divisor (resp. minimally discrepant divisor)
}over $X$ as a divisor $E$ over $X$ with $a(E)=0$ (resp. $a(E)+1=\mld(X)$
). 
\end{defn}
If $X$ is canonical, then there exist only finitely many crepant
divisors over $X$ and they all appear on $Y$ for an arbitrary log
resolution $Y\to X$. The same is true for minimally discrepant divisors
if $X$ is log terminal. 
\begin{notation}
We denote the number of crepant divisors over $X$ by $\gamma(X)$
and the one of minimally discrepant divisors by $\delta(X)$ . 
\end{notation}
We note that the minimal log discrepancy is stable under extensions
of the base field. In particular, 
\[
\mld(X)=\mld(X_{\overline{F}}).
\]
However $\delta(X)$ and $\gamma(X)$ can change by extensions of
the base field. This happens when relevant divisors over $X$ are
not geometrically irreducible. However, if $F'/F$ is a sufficiently
large finite extension, then $\delta(X_{F'})=\delta(X_{\overline{F}})$
and $\gamma(X_{F'})=\gamma(X_{\overline{F}})$.

\section{Rational points on singular Fano varieties\label{sec:Rational-points-on}}

In this section, we discuss the distribution of rational points on
(possibly singular) Fano varieties, especially Manin's conjecture
\cite{Franke:1989go} and several variants of it.

\subsection{Smooth varieties}

Let $X$ be a projective variety over a number field $K$ and $D$
a metrized $\QQ$-divisor on it. For any subset $U\subset X(K)$ and
for $B>0$, let
\[
\cN_{U,D}(B):=\sharp\{x\in U\mid H_{D}(x)\le B\}.
\]
When there is no confusion, we omit the subscript $D$. If $\cN_{U}(B)<\infty$
for all $B$, it is natural to study its asymptotic behavior as $B$
tends to infinity. If $D$ is ample, then the finiteness of $\cN_{U}(B)$
holds for any $U$ and $B$. As a slight generalization, if $D$ is
big, then there exists a Zariski open subset $V\subset X$ such that
$\cN_{V(K)}(B)$ is finite for every $B>0$. 

Manin's conjecture is about the case where $X$ is a smooth Fano variety,
that is, the anti-canonical divisor $-K_{X}$ is ample, and $D=-K_{X}$.
We state a variant of it which uses cothin subsets.
\begin{defn}
\label{def:thin-thick}For a $K$-variety $X$, a subset $A\subset X(K)$
is said to be \emph{thin }if there exists a generically finite morphism
$Y\to X$ of $K$-varieties without admitting a rational section $X\dashrightarrow Y$
such that $A$ is contained in the image of $Y(K)$. A subset of $X(K)$
is said to be \emph{cothin }if it is the complement of a certain thin
subset. 
\end{defn}
The following is a variant of Manin's conjecture:
\begin{conjecture}
\label{conj:BM}Let $X$ be a smooth Fano variety over $K$. If $K$
is sufficiently large, then for a sufficiently small cothin subset
$U\subset X(K)$ and for an arbitrary adelic metric on $-K_{X}$,
we have 
\[
\cN_{U}(B)\sim CB(\log B)^{\rho(X)-1}.
\]
Here $\rho(X)$ is the Picard number, the rank of the Néron--Severi
group. 
\end{conjecture}
The precise meaning of ``if $K$ is sufficiently large'' is that
there exists a finite extension $K_{0}/K$ such that for all finite
extension $L/K_{0}$, the relevant statement is true after the base
change to $L$. This condition hopefully assures that $X(K)$ is not
too small (at least not empty) and $\rho(X)$ is equal to $\rho(X_{\overline{K}})$
so that the subtleness of the exponent of the log factor would be
diminished.

Originally the $K$-point set of a Zariski open subset was considered
rather than a cothin subset (see \cite{Franke:1989go,MR1032922}).
However that version has a counter-example found by Batyrev--Tschinkel
\cite{Batyrev:1996ut}. Peyre \cite{MR2019019} then suggested to
consider cothin subsets as one of possible ways to remedy the pathological
situation. Le Rudulier \cite{LeRudulierSurface} and Browning--Loughran
\cite{Browning:2013aa} more seriously studied this modification of
the conjecture. 

Next we consider a generalization by Batyrev--Manin \cite{MR1032922}.
Let $X$ be a smooth projective variety over $K$, $\NS(X)$ its Néron-Severi
group, $\NS(X)_{\RR}:=\NS(X)\otimes\RR$ and $\Eff(X)\subset\NS(X)_{\RR}$
the pseudo-effective cone, the closure of the cone generated by the
classes of effective divisors. We suppose that $K_{X}\notin\Eff(X)$,
which is equivalent to that $X$ is uniruled \cite{Boucksom:2013er}.
Let $D$ be a big $\QQ$-divisor, that is, its class is in the interior
of $\Eff(X)$ (see \cite[p. 147]{MR2095471}). 
\begin{defn}
We define $\alpha(D)\in\RR$ by the condition
\[
\alpha(D)\cdot[D]+[K_{X}]\in\partial\Eff(X),
\]
where $\partial\Eff(X)$ denotes the boundary of $\Eff(X)$. We denote
the point $\alpha(D)\cdot[D]+[K_{X}]$ by $\partial(D)$.
\end{defn}
Batyrev--Manin \cite{MR1032922} conjectures that the cone $\Eff(X)$
is polyhedral around $\partial(D)$. 
\begin{defn}
Under this conjecture, we let $\beta(D)$ be the codimension in $\NS(X)_{\RR}$
of the minimal face of $\Eff(X)$ containing $\partial(D)$. 
\end{defn}
The following is a variant of their conjecture:
\begin{conjecture}
\label{conj:BM general D}Let $X$ be a smooth projective variety
over a number field $K$ with $K_{X}\notin\Eff(X)$ and $D$ a big
divisor on $X$. If $K$ is sufficiently large, then for a sufficiently
small cothin subset $U\subset X(K)$, giving an arbitrary adelic metric
to $D$, we have 
\[
\cN_{U,D}(B)\sim CB^{\alpha(D)}(\log B)^{\beta(D)-1}.
\]

\end{conjecture}
If $X$ is a Fano variety and $D=-K_{X}$, we recover Conjecture \ref{conj:BM}.

\subsection{Log terminal Fano varieties}

When the given Fano variety has singularities, we cannot generally
expect the asymptotic formula of the same form and need to take contribution
of singularities into account. Batyrev--Tschinkel \cite{MR1679843}
explored Fano varieties having canonical singularities. The following
conjecture is a variant of a special case of their conjectural formula
(page 323, \emph{loc.cit.}).
\begin{conjecture}
\label{conj:BT}Let $X$ be a log terminal Fano variety and $\gamma(X)$
be the number of crepant divisors over $X$. If $X$ is canonical
and $K$ is sufficiently large, then for a sufficiently small cothin
$U\subset X(K)$, 
\[
\cN_{U}(B)\sim CB(\log B)^{\rho(X)+\gamma(X)-1}.
\]
If $X$ is not canonical and $K$ is sufficiently large $K$, then
for any cothin $U\subset X(K)$, we have 
\[
\cN_{U}(B)\sim CB^{\alpha}(\log B)^{\beta}
\]
with $\alpha>1$ and $\beta\ge0$. 
\end{conjecture}
Actually one can deduce this conjecture (at least with $U$ sufficiently
small for the second assertion) from Conjecture \ref{conj:BM general D}
and the following lemma.
\begin{lem}
Let $f:Y\to X$ be a log resolution of a log terminal Fano variety
$X$. We have that 
\[
\alpha(-f^{*}K_{X})\begin{cases}
=1 & (X\text{: canonical})\\
>1 & (\text{otherwise})
\end{cases}.
\]
Moreover, if $X$ is canonical, $K$ is sufficiently large and $\Eff(X)$
is polyhedral around $\partial(-f^{*}K_{X})$, then the minimal face
of $\Eff(X)$ containing $\partial(-f^{*}K_{X})$ has codimension
$\rho(X)+\gamma(X)$.\end{lem}
\begin{proof}
Let $f:Y\to X$ be a log resolution and write 
\[
K_{Y}-f^{*}K_{X}=\sum_{i=1}^{l}a_{i}E_{i}
\]
where $E_{i}$ are exceptional prime divisors and $a_{i}\in\QQ$.
We first show the first assertion in the case where $X$ is canonical
(that is, for every $i$, $a_{i}\ge0$). Since $K_{Y}$ and $-f^{*}K_{X}$
are both big, the invariant $\alpha(-f^{*}K_{X})$ is defined. The
class 
\[
B:=[K_{Y}]-[f^{*}K_{X}]=\sum a_{i}[E_{i}]
\]
clearly belongs to $\Eff(X)$. To see that it lies in the boundary
$\partial\Eff(X)$, it suffices to show that for an ample class $A$
and for any $\epsilon>0$, $B-\epsilon A$ does not belong to $\Eff(Y)$.
Let $U\subset X$ be the locus where $f$ is an isomorphism. Its complement
$X\setminus U$ has codimension $\ge2$. If $C\subset X$ is a smooth
curve obtained as the intersection of $\dim X-2$ general hyperplane
sections, then $C$ is contained in $U$ and its lift $\tilde{C}$
to $Y$ does not meet the exceptional locus of $f$. Therefore, the
intersection number, $B\cdot\tilde{C}$, is negative. However, $\tilde{C}$
is a member of a family of curves covering general points of $Y$,
any class in $\Eff(Y)$ should intersect $\tilde{C}$ with a non-negative
intersection number. We have proved the first assertion in this case.

Next consider the case where $X$ is not canonical. At least one of
$a_{i}$ is negative. The intersection of $\Eff(Y)$ and the linear
space spanned by $[E_{i}]$, $i=1,\dots,l$ is the cone spanned by
$[E_{i}]$, $i=1,\dots,l$. This shows that the divisor class $B$
is not in $\Eff(Y)$. 

As for the second assertion, from \cite[Lem. 4.1]{Chen:2014fh}, each
ray $\RR_{\ge0}[E_{i}]$ is an extremal ray of $\Eff(Y)$. Moreover
they are linearly independent. It follows that the minimal face of
$\Eff(Y)$ containing $B$ is the cone generated by those classes
$[E_{i}]$ with $a_{i}>0$, whose codimension is $\rho(X)+\gamma(X)$. 
\end{proof}

\subsection{Toric Fano varieties}

When the given Fano variety is not canonical, then it seems generally
difficult to determine the exponents of $B$ and $\log B$ in the
expected asymptotic formula. In the case of toric varieties, however,
we can obtain a natural upper bound. We mean by a toric variety an
equivariant compactification of the torus $T=(\GG_{m,K})^{d}$.\footnote{Batyrev and Tschinkel \cite{MR1423638,MR1620682} considers a more
general notion of algebraic tori, which become isomorphic to $(\GG_{m,K})^{d}$
after a finite extension of the base field $K$.} 
\begin{lem}
\label{lem:toric-alpha}Let $X$ be a log terminal toric Fano variety
which is not canonical (that is, $0<\mld(X)<1$) and $f:Y\to X$ a
toric resolution. Then 
\[
\alpha(-f^{*}K_{X})\le\frac{1}{\mld(X)}.
\]
Moreover, if the equality holds, then the minimal face of $\Eff(Y)$
containing $\partial(-f^{*}K_{X})$ has codimension $\le\delta(X)$.\end{lem}
\begin{proof}
Let $D_{i}$, $i\in I$ be the torus invariant prime divisors on $X$
and $\tilde{D}_{i}$ their respective strict transforms on $Y$. Let
$E_{j}$, $j\in J$ be the exceptional prime divisors of $f$. Canonical
divisors of $X$ and $Y$ have the natural expressions, 
\[
K_{X}=-\sum_{i\in I}D_{i}\text{ and }K_{Y}=-\sum_{i\in I}\tilde{D}_{i}-\sum_{j\in J}E_{j}.
\]
We write
\[
K_{Y}=f^{*}K_{X}+\sum_{j\in J}a_{j}E_{j}
\]
with $a_{j}>-1$. For a real number $c$, we have 
\begin{equation}
-c(f^{*}K_{X})+K_{Y}=\sum_{i}(c-1)\tilde{D}_{i}+\sum_{j}\left(c(1+a_{j})-1\right)E_{j}.\label{eq:pseudo-eff}
\end{equation}
For $c=1/\mld(X)>1$, this divisor is clearly effective, which shows
the first assertion.

Suppose that the equality, $\alpha(-f^{*}K_{X})=1/\mld(X)$, holds.
Let 
\[
J_{0}:=\{j\mid a_{j}+1=\mld(X)\},
\]
whose cardinality is by definition $\delta(X)$. In the right hand
side of (\ref{eq:pseudo-eff}), $D_{i}$, $i\in I$ and $E_{j}$,
$J\setminus J_{0}$ have positive coefficients and $E_{j}$, $j\in J_{0}$
have zero coefficients. We note that $\Eff(Y)$ is generated by $[D_{i}]$,
$i\in I$ and $[E_{j}]$, $j\in J$. It follows that the minimal face
containing $\partial(-f^{*}K_{X})$ is the cone generated by $[D_{i}]$,
$i\in I$ and $[E_{j}]$, $j\in J\setminus J_{0}$. Its codimension
is clearly at most $\sharp J_{0}=\delta(X)$. \end{proof}
\begin{rem}
The above proof shows that if we use the group of the torus invariant
divisors (NOT modulo linear equivalence) rather than the Néron--Severi
group and the cone of the effective torus invariant divisors, then
the equalities in the inequalities of the lemma hold. The numbers,
$1/\mld(X)$ and $\delta(X)$, are the exponents of an asymptotic
formula expected from our heuristics (Section \ref{sec:Manin-vs.-Malle})
and a generalization of Malle's conjecture (Conjecture \ref{conj:general Malle}).
On the other hand, these numbers are generally wrong ones from the
viewpoint of the Batyrev--Manin conjecture as we see below. Thus,
there seems to be some subtleness or obstruction so that our heuristics
do not work when $X$ is not canonical.\end{rem}
\begin{cor}
\label{cor:toric-strict-ineq}Suppose that $X$ is a toric Fano variety
of Picard number one with $0<\mld(X)<1$. Suppose that there exists
a cone of the defining fan which contains all singular cones of the
fan as its faces. Then, for some real numbers $a$ and $b$ with $1<a<1/\mld(X)$
and $b\ge0$, we have 
\[
\cN_{T(K),-K_{X}}(B)\sim CB^{a}(\log B)^{b}
\]
with $T$ the maximal torus.\end{cor}
\begin{proof}
Since Conjecture \ref{conj:BM general D} holds for smooth toric varieties
with $U=T(K)$, we need to show that the strict inequality holds in
Lemma \ref{lem:toric-alpha}. We follows the notation of Lemma \ref{lem:toric-alpha}
and its proof. The assumption on the Picard number shows that there
are exactly $d+1$ torus invariant prime divisors on $X$, say $D_{0},\dots D_{d}$
with $d=\dim X$. Without loss of generality, we may suppose that
the cone spanned by the rays corresponding to $D_{1},\dots,D_{d}$
contains all singular cones, and that the chosen toric resolution
$f:Y\to X$ subdivides only this cone and its faces. Then $D_{0}$
is linearly equivalent to $\sum_{i=1}^{d}a_{i}\tilde{D}_{i}+\sum_{j\in J}b_{j}E_{j}$
with $a_{i},b_{j}>0$. Therefore the class of $D_{0}$ is in the interior
of $\Eff(X)$. Therefore, for $c=1/\mld(X)$, the divisor (\ref{eq:pseudo-eff})
is not in the boundary of $\Eff(X)$. This shows the desired inequality. 
\end{proof}

\section{Height zeta functions\label{sec:Height-zeta-functions}}

The height zeta function is an analytic tool to study the distribution
of rational points. Suppose that a $K$-variety $X$ is given a height
function $H:X(K)\to\RR_{>0}$. For $U\subset X(K),$ we write again
\[
\cN_{U}(B)=\sharp\{x\in U\mid H(x)\le B\}.
\]
The \emph{height zeta function} $Z_{U}(s)$ is defined as the (generalized)
Dirichlet series
\[
Z_{U}(s):=\sum_{x\in U}H(x)^{-s}.
\]
If the height function is associated to a metrized ample divisor,
then $\cN_{U}(B)$ has the polynomial growth and hence $Z_{U}(s)$
is convergent for some half-plane $\Re(s)>a$. The infimum of such
$a$ is called the \emph{abscissa of convergence}. 

For heuristic arguments, it is convenient to assume meromorphic continuation
beyond the abscissa of convergence, although it is generally difficult
to prove it. For the sake of generality (we will consider the same
condition for extended $V$-discriminant zeta functions), we consider
the Dirichlet series 
\[
f(s):=\sum_{x\in I}F(x)^{-s}
\]
associated to an arbitrary function $F:I\to\RR_{>0}$ on a countable
set $I$.
\begin{condition}
\label{cond:merom-continuation}The series $f(s)$ converges for some
half-plane $\Re(s)>a$. If the abscissa of convergence, say $\alpha$,
is not $-\infty$, then $f(s)$ extends to a meromorphic function
along the line $\Re(s)=\alpha$ which is holomorphic except at $s=\alpha$. 
\end{condition}
Supposing that $f(s)$ satisfies the condition, let $\beta$ be the
order of the pole at $s=\alpha$ and 
\[
C:=\frac{1}{\Gamma(\beta)\alpha}\lim_{s\to\alpha}\frac{f(s)}{(s-\alpha)^{\beta}}.
\]
Then a version of Tauberian theorem shows
\[
\sharp\{x\in I\mid F(x)\le B\}\sim CB^{\alpha}(\log B)^{\beta-1}\quad(B\to\infty).
\]

\section{Peyre's constants for projective spaces\label{sec:Peyre's-constants-for}}

Peyre \cite{MR1340296} refined Manin's conjecture by specifying the
constant factor $C$ in terms of Tamagawa measures. We only need it
for projective spaces. In this case, his refinement is actually a
theorem and generalizes the earlier result of Schanuel \cite{MR0162787}
as the case of a special metric on the anti-canonical divisor. It
is essential in our work to consider various metrics. 

Let $X=\PP_{K}^{d}$ be the projective space over a number field $K$.
We suppose that the anti-canonical divisor $-K_{X}$ is metrized.
For each $\fp$, the metric defines a measure $\mu_{\fp}$ on the
$K_{\fp}$-analytic manifold $X(K_{\fp})$ as follows. For local coordinates
$x_{1},\dots,x_{d}$ on $X(K_{\fp})$, consider the differential form
\[
\omega=\left\Vert \frac{\partial}{\partial x_{1}}\wedge\cdots\wedge\frac{\partial}{\partial x_{d}}\right\Vert _{\fp}dx_{1}\wedge\cdots\wedge dx_{d}.
\]
The measure is locally defined by 
\[
\mu_{\fp}(A):=\int_{A}\omega.
\]
Here the integral is defined with respect to the normalized Haar measure
on $(K_{\fp})^{d}$. The local measures glue together and give a measure
on the whole space $X(K_{\fp})$. If the fixed model of $-K_{X}$
is the anti-canonical divisor $-K_{\PP_{\cO_{S}}^{d}/\cO_{S}}$ of
$\PP_{\cO_{S}}^{d}$ over $\cO_{S}$, then for $\fp\in S^{c}$, 
\[
\mu_{\fp}(X(K_{\fp}))=\frac{\PP_{\cO_{\fp}}^{d}(\kappa_{\fp})}{N_{\fp}^{d}}=\sum_{i=0}^{d}N_{\fp}^{-i},
\]
where $\kappa_{\fp}$ is the residue field of $K_{\fp}$ and $N_{\fp}:=\sharp\kappa_{\fp}$.
More generally, suppose that the model of $-K_{X}$ is 
\[
-K_{\PP_{\cO_{S}}^{d}/\cO_{S}}+\sum_{\fp\in S^{c}}a_{\fp}\PP_{\kappa_{\fp}}^{d},
\]
where $a_{\fp}$ are rational numbers and almost all of them are zero,
and $\PP_{\kappa_{\fp}}^{d}$ are regarded as prime divisors on $\PP_{\cO_{S}}^{d}$.
Then, for $\fp\in S^{c}$, the $\fp$-adic metric of $-K_{X}$ is
simply $N_{\fp}^{-a_{\fp}}$ times the $\fp$-adic metric induced
from the model $-K_{\PP_{\cO_{S}}^{d}/\cO_{S}}$ and 
\[
\mu_{\fp}(X(K_{\fp}))=N^{-a_{\fp}}\sum_{i=0}^{d}N_{\fp}^{-i}.
\]
 
\begin{thm}[{\cite[Cor. 6.2.17, 6.2.18]{MR1340296}. See also \cite{LeRudulierLine}.}]
\label{thm:Peyre}Let $c_{K}$ be the residue of the Dedekind zeta
function $\zeta_{K}(s)$ at $s=1$, let $\Disc_{K}$ be the absolute
Discriminant of $K$ and let 
\[
C_{\Peyre}:=\frac{c_{K}}{|\Disc_{K}|^{d/2}}\prod_{\fp\in M_{K,f}}\left(1-\frac{1}{N_{\fp}}\right)\mu_{\fp}(X(K_{\fp}))\times\prod_{\fp\in M_{K,\infty}}\mu_{\fp}(X(K_{\fp})).
\]
For any Zariski open subset $U\subset X$, we have
\[
\cN_{U(K)}(B)\sim C_{\Peyre}B.
\]
\end{thm}
\begin{rem}
That we can take any Zariski open subset $U$ is a consequence of
the property ($E_{V}$) in \cite{MR1340296}.
\end{rem}

\section{Quotients of projective spaces\label{sec:Quotients-of-projective}}

In this section, we consider a special class of log terminal Fano
varieties. Let $G$ be a finite group and $V=\AA_{K}^{d}$ an affine
space endowed with a linear faithful $G$-action. The action extends
to the projective space $\overline{V}=\PP_{K}^{d}$. We write the
associated quotient varieties as 
\[
X:=V/G\text{ and }\overline{X}:=\overline{V}/G.
\]
Suppose that the quotient map $\overline{V}\to\overline{X}$ is étale
in codimension one. It is equivalent to the condition that $G\subset\GL(V)$
has no pseudo-reflection and no element of $G\setminus\{1\}$ acts
on $V$ by a scalar multiplication. Then the pull-back of $-K_{\overline{X}}$
to $\overline{V}$ is $-K_{\overline{V}}$ and hence $\overline{X}$
is a log terminal Fano variety. If $G$ is abelian and $K$ is so
large to have enough roots of unity, then $\overline{X}$ is a toric
variety. The Picard number of $\overline{X}$ is one. Moreover its
defining fan satisfies the condition in Corollary \ref{cor:toric-strict-ineq}.
\begin{rem}
If some element of $G\setminus\{1\}$ acts by a scalar multiplication,
then we can reduced to the situation where this does not happen, simply
replacing $V=\AA_{K}^{d}$ with $\AA_{K}^{d+1}$ by adding an additional
variable on which $G$ acts trivially. 
\end{rem}
We recall a representation-theoretic formula for the minimal log discrepancies
of $X=V/G$ and $\overline{X}=\PP_{K}^{d}/G$. Since the minimal log
discrepancy is stable under base field extensions, we consider the
induced $G$-action on $V_{\overline{K}}$. We fix a primitive $\sharp G$-th
root of unity, $\zeta\in\overline{K}$. We can diagonalize each $g\in G\subset\GL(V_{\overline{K}})$
and write 
\[
g=\diag(\zeta^{a_{1}},\dots,\zeta^{a_{d}})\quad(0\le a_{i}<\sharp G).
\]

\begin{defn}[\cite{MR1463181}]
We define the \emph{age }of $g$, denoted by $\age(g)$, to be
\[
\frac{1}{\sharp G}\sum_{i=1}^{d}a_{i}.
\]
We then define the \emph{age }of $G$, denoted $\age(G)$, as the
minimum of $\age(g)$, $g\in G\setminus\{1\}$. 
\end{defn}
Note that $\age(g)$ depends on the choice of $\zeta$, but $\age(G)$
does not. The following proposition gives a simple representation
theoretic description of the minimal log discrepancy of $X=V/G$ and
the one of $\overline{X}=\PP_{K}^{d}/G$.
\begin{prop}
\label{prop:disc age-1}We have 
\[
\mld(\overline{X})=\mld(X)=\age(G).
\]
\end{prop}
\begin{proof}
The right equality should be well-known for specialists. See \cite{MR2271984}
for instance. We prove the left one. Obviously $\mld(\overline{X})\ge\mld(X)$.
To see the converse, we consider the embedding 
\[
V=\AA_{K}^{d}\hookrightarrow\tilde{V}=\AA_{K}^{d+1},\,(x_{1},\dots,x_{d})\mapsto(x_{1},\dots,x_{d},1)
\]
and the induced embedding $\GL(V)\hookrightarrow\GL(\tilde{V})$.
We denote the image of $G$ in $\GL(\tilde{V})$ by $\tilde{G}$.
This subgroup $\tilde{G}$ does not have pseudo-reflection either
and has the same age as $G$. Now $\overline{X}$ is the quotient
of $(\tilde{V}\setminus\{0\})/\tilde{G}$ by the natural $\GG_{m}$-action.
Therefore 
\begin{align*}
\mld(\overline{X}) & =\mld(\tilde{V}\setminus\{0\}/\tilde{G})\\
 & \le\mld(\tilde{V}/\tilde{G})\\
 & =\age(\tilde{G})\\
 & =\age(G)\\
 & =\mld(X).
\end{align*}

\end{proof}
Let $\Conj(G)$ be the set of conjugacy classes of $G$. We can make
$\Gal(\overline{\QQ}/\QQ)$ act on $\Conj(G)$ via the cyclotomic
character \cite{MR2068887}. 
\begin{defn}
We define the \emph{set of} \emph{$K$-conjugacy classes }as the quotient
of $\Conj(G)$ by the action of $\Gal(\overline{K}/K)\subset\Gal(\overline{\QQ}/\QQ)$
and denote it by $\KConj(G)$. A \emph{$K$-conjugacy class} of $K$
then consists of elements of $G$ giving the same element of $\KConj(G)$.
\end{defn}
If $K$ is sufficiently large, then the action of $\Gal(\overline{K}/K)$
on $\Conj(G)$ is trivial and $\Conj(G)=\KConj(G)$. Therefore the
notion of $K$-conjugacy does not play an important role in our formulation
of conjectures. However, the following proposition provides one more
evidence for the close relation between Manin's and Malle's conjectures.
\begin{prop}
\label{prop:div min}For $g\in G\setminus\{1\}$ with $\age(g)=\age(G)$,
if $[g]=[g']$ in $\KConj(G)$, then $\age(g)=\age(g')$. Moreover
the number of $K$-conjugacy classes of $G$ with minimal age is equal
to the number of minimally discrepant divisors, $\delta(X)$. \end{prop}
\begin{notation}
We call an element of $G$ or a ($K$-)conjugacy class of $G$ the
\emph{youngest }if its age is equal to the one of $G$. We denote
the number of the youngest $K$-conjugacy classes of $G$ by $\upsilon(G)$. 
\end{notation}
The second assertion of the proposition reads
\[
\delta(X)=\upsilon(G).
\]

\begin{proof}[Proof of the proposition]
We first sketch the outline of the proof. A version of the McKay
correspondence says that there exists a one-to-one correspondence
between minimally discrepant divisors over $X_{\overline{K}}$ and
the youngest conjugacy classes of $G$ (this already proves the proposition
when $K$ is sufficiently large). We can make the correspondence compatible
with the $\Gal(\overline{K}/K)$-actions and the proposition follows. 

To make arguments more precise, we use results from \cite{MR2271984}.
However the reader may skip this part, since it is not related to
the rest of the paper. Let $X_{\sing}$ be the singular locus of $X$,
$f:Y\to X$ a log resolution which is an isomorphism over $X\setminus X_{\sing}$.
Let $J_{\infty}X$ and $J_{\infty}Y$ be the arc spaces of $X$ and
$Y$ and let $J_{\infty}'Y$ be the preimage of $X_{\sing}$ by $J_{\infty}Y\to Y\to X$.
Let $\cX:=[V/G]$ be the quotient stack and $\cJ_{\infty}\cX$ its
twisted arc space and $\cJ'_{\infty}\cX$ the preimage of $X_{\sing}$
by $\cJ_{\infty}\cX\to\cX\to X$. There exists a canonical one-to-one
correspondence of the points sets,
\begin{equation}
|J_{\infty}'Y|\leftrightarrow|\cJ_{\infty}'\cX|,\label{eq:corr-arc}
\end{equation}
outside measure zero subsets with respect to the motivic measures.
The same is true for $\overline{K}$-point sets of $J_{\infty}'Y$
and $\cJ'_{\infty}\cX$, and the correspondence is $\Gal(\overline{K}/K)$-equivariant. 

Once we fix a primitive $\sharp G$-th root of unity $\zeta\in\overline{K}$,
there exists a one-to-one correspondence between connected components
of $|(\cJ_{\infty}'\cX)_{\overline{K}}|$ with respect to the Zariski
topology and $\Conj(G)$, which respects $\Gal(\overline{K}/K)$-actions.
Therefore there exists a one-to-one correspondence between connected
components of $|\cJ_{\infty}'\cX|$ and $\KConj(G)$. We denote by
$C_{1},\dots,C_{l}$ the connected component of $|\cJ_{\infty}'\cX|$
corresponding to the youngest $K$-conjugacy classes. 

Let $E_{1},\dots,E_{m}$ be the minimally discrepant divisors on $Y$
and $\tilde{E}_{1},\dots,\tilde{E}_{m}$ their preimages in $J_{\infty}'Y$.
The correspondence (\ref{eq:corr-arc}) respects suitably normalized
motivic measures, which take values in a semiring $\mathfrak{R}^{1/r}$
having the dimension function $\dim:\mathfrak{R}^{1/r}\to\QQ$. If
we denote these normalized measures by $\mu_{Y}^{\nu}$ and $\mu_{\cX}^{\nu}$,
then we have 
\[
\mu_{Y}^{\nu}(J_{\infty}'Y)\equiv\sum_{i}\mu_{Y}^{\nu}(\tilde{E}_{i})\equiv\mu_{\cX}^{\nu}(\cJ_{\infty}'\cX)\equiv\sum_{j}\mu_{\cX}^{\nu}(C_{j}),
\]
where the congruence are taken modulo lower dimensional elements.
This shows that for each $i$, the generic point of $\tilde{E}_{i}$
corresponds to the generic point of some $C_{j}$. This completes
the proof.
\end{proof}

\section{A simple correspondence of points\label{sec:A-simple-correspondence}}

In this section, $V$ is an arbitrary variety over a base field $F$
of characteristic zero with a faithful $G$-action and $X=V/G$ is
the associated quotient variety. 
\begin{defn}
When $\Spec L$ is an étale $G$-torsor over $\Spec F$, we call the
$K$-algebra $L$ with a $G$-action a $G$-\emph{ring (over $F$).
}An \emph{isomorphism} of two $G$-rings is a $G$-equivariant $F$-isomorphism.
When the underlying ring of a $G$-ring $L$ is a field, we call $L$
a $G$\emph{-field. }We denote the set of $G$-rings modulo isomorphism
by $G\Tor(F)$ and the one of $G$-fields by $G\Fie(F)$. We sometimes
call a $G$-field a \emph{large }$G$-field, distinguishing it from
\emph{small} $G$-fields introduced later.
\end{defn}

\begin{defn}
For $L\in G\Tor(F)$ and a $K$-variety $Y$ with $G$-action, a \emph{$G$-equivariant
$L$-point} of $Y$ is simply a $G$-equivariant morphism $\Spec L\to Y$.
We denote the set of $G$-equivariant $L$-points of $Y$ by $Y(L)^{G}$.
\end{defn}
Let $V$ be an $F$-variety with a faithful $G$-action and let $X:=V/G$.
For $L\in G\Tor(F)$, taking the $G$-quotients of the source and
the target defines a natural map 
\[
V(L)^{G}\to X(F).
\]
The group $\Aut(L)$ of automorphisms $L$ (as a $G$-ring) acts on
$V(L)^{G}$ and the above map factors through the quotient set $V(L)^{G}/\Aut(L)$.
Let $V_{\ur}$ and $X_{\ur}$ be the unramified loci of $V\to X$
in $V$ and $X$ respectively. We obtain the map 
\[
V_{\ur}(L)^{G}/\Aut(L)\to X_{\ur}(F),
\]
which is easily seen to be injective. Conversely, given $x\in X_{\ur}(F)$,
the morphism 
\[
\Spec L:=\Spec F\times_{X}V\to V
\]
defines a $G$-equivariant $L$-point. We thus obtain a bijection
\[
\bigsqcup_{L\in G\Tor(F)}V_{\ur}(L)^{G}/\Aut(L)\to X_{\ur}(F).
\]

\begin{lem}
A point $x\in X_{\ur}(F)$ is in the image of $V_{\ur}(L)^{G}/\Aut(L)$
for $L\in G\Fie(F)$ if and only if $x$ is not in the image of $(V/H)(F)$
for any proper subgroup $H\subsetneq G$.\end{lem}
\begin{proof}
If $x\in X_{\ur}(F)$ is the image of $(V/H)(F)$ for $H\subsetneq G$,
then $\Spec F\times_{X}(V/H)$ is the disjoint union of $\sharp(G/H)$
copies of $\Spec F$ and hence
\[
\Spec F\times_{X}V=(\Spec F\times_{X}(V/H))\times_{V/H}V
\]
is not connected. Conversely, if $\Spec L=\Spec F\times_{X}V$ is
not connected, then let $H$ be the stabilizer of a component. Taking
the $H$-quotients, we get an $F$-point of $V/H$ which maps to $x$.\end{proof}
\begin{defn}
When $x\in X_{\ur}(F)$ satisfies one of the equivalent conditions
in the last lemma, we say that $x$ is a \emph{primitive} $F$-\emph{point
}of $X$. We denote the set of primitive $F$-points by $X_{\prim}(F)$. 
\end{defn}
When $F$ is a number field, the subset $X_{\prim}(F)\subset X(F)$
is the cothin subset obtained by removing all ``natural''\emph{
}thin subsets. This is why we will below think of $X_{\prim}(F)$
as a candidate of a \emph{sufficiently small }cothin subset below. 

By construction, we have a natural bijection
\begin{equation}
\bigsqcup_{L\in G\Fie(F)}V(L)^{G}/\Aut(L)\to X_{\prim}(F).\label{eq:pt-corr}
\end{equation}

We need the following fact later.
\begin{lem}
\label{lem:free-action}For $L\in G\Fie(F)$, the $\Aut(L)$-action
on $V_{\ur}(L)^{G}$ is free.\end{lem}
\begin{proof}
Let $g\in G$, $y\in V_{\ur}(L)^{G}$ and $x:=\pi(y)\in X(K)$. If
$g(y)=y$, then $y:\Spec L\to V_{\ur}$ factors through $\Spec L^{\left\langle g\right\rangle }$,
where $L^{\left\langle g\right\rangle }$ is the invariant subfield
by the cyclic group $\left\langle g\right\rangle $. There exists
a morphism 
\[
\Spec L^{\left\langle g\right\rangle }\to\Spec L\cong\Spec F\times_{x,X}V_{\ur}
\]
which is a section of the natural morphism $\Spec L\to\Spec L^{\left\langle g\right\rangle }$.
Therefore $\Spec L$ is the disjoint union of $\sharp\left\langle g\right\rangle $
copies of $\Spec L^{\left\langle g\right\rangle }$, which contradicts
the assumption $L\in G\Fie(F)$. 
\end{proof}

\section{Untwisting\label{sec:Untwisting}}

The \emph{untwisting }technique allows us to reduce the study of $G$-equivariant
$L$-points on affine spaces for $L\in G\Fie(K)$ to the study of
ordinary $K$-points on affine spaces. It was introduced by Denef--Loeser
\cite{MR1905024} over the ring of power series $k[[t]]$ with $k$
a field of characteristic zero containing enough roots of unity in
an explicit way. The author \cite{wild-p-adic,Yasuda:2013fk,Yasuda:2014fk2}
generalized it to the integer ring of every local field in a more
intrinsic way, which admits a literal translation to a more general
situation. We basically follows the presentation in \cite{Yasuda:2014fk2}.

\subsection{Untwisting over a general base}

Let $\cO$ be an integrally closed noetherian domain of dimension
$\le1$ with the fraction field $F=\mathrm{frac}(\cO)$ of characteristic
zero. Suppose that the affine space $W=\AA_{\cO}^{d}$ has an $\cO$-linear
$G$-action. We write $W=\Spec S^{\bullet}\cF$, where $\cF$ is a
free $\cO$-module of rank $d$ and $S^{\bullet}\cF$ is its symmetric
algebra over $\cO$. For $L\in G\Tor(F)$, let $\cO_{L}$ be the integral
closure of $\cO$ in $L$. The natural $G$-action on $L$ restricts
to $\cO_{L}$. 
\begin{defn}
The associated \emph{tuning module }is defined by 
\[
\Xi_{L}:=\cHom_{\cO}^{G}(\cF,\cO_{L}),
\]
the sheaf of $G$-equivariant $\cO$-linear maps. 
\end{defn}
It turns out that $\Xi_{L}$ is a locally free $\cO$-module of rank
$d$ (see \cite{Wood-Yasuda-I,Yasuda:2013fk}). Let 
\[
\Theta_{L}:=\cHom_{\cO_{T}}(\Xi_{L},\cO).
\]

\begin{defn}
The \emph{untwisting variety }and the \emph{pre-untwisting variety
}of $W$ with respect to $L$ are respectively defined as 
\begin{gather*}
W^{|L|}:=\Spec S^{\bullet}\Theta,\\
W^{\left\langle L\right\rangle }:=W\otimes_{\cO}\cO_{L}.
\end{gather*}

\end{defn}
The untwisting variety $W^{|L|}$ is a vector bundle of rank $d$
over $\Spec\cO$. If $\cO$ is a principal ideal domain, then $W^{|L|}$
is $\cO$-isomorphic to the original affine space $W$. We exploit
this fact in later sections.

There exists a commutative diagram of natural morphisms:
\begin{equation}
\xymatrix{ & W^{\left\langle L\right\rangle }\ar[dl]\ar[dr]\\
W\ar[dr] &  & W^{|L|}\ar[dl]\\
 & W/G
}
\label{eq:untw diag}
\end{equation}
The construction of this diagram is compatible with the base change
by a flat morphism $\cO'\to\cO$ with $\cO'$ another integrally closed
noetherian domain of dimension $\le1$. The two upper arrows are étale
over the generic point $\Spec F$ of $\Spec\cO$. 

The diagram induces a one-to-one correspondence between $G$-equivariant
$\cO_{L}$-points of $W$ and $\cO$-points of $W^{|L|}$:
\[
W(\cO_{L})^{G}\leftrightarrow W^{|L|}(\cO)
\]
At the cost of a ``little'' twist of the target space $W$, this
correspondence allows us to reduce the study of equivariant points
to the one of ordinary ones. There exist natural $\Aut(L)$-actions
on the sets $W(\cO_{L})^{G}$ and $W^{|L|}(\cO)$ and the above correspondence
is compatible with these actions. The correspondence is also compatible
with the natural maps $W(\cO_{L})^{G}\to(W/G)(\cO)$ and $W^{|L|}(\cO)\to(W/G)(\cO)$. 
\begin{rem}
\label{rem:tw-untw}The twist of the target space (how different $W$
and $W^{|L|}$ are) causes a subtle and hard problem in our study
relating the distribution of rational points and the one of number
fields. That is why we need to appeal to a heuristic argument below.
When $F=K$ is a number field, even if $\cO=\cO_{S}$ is a principally
ideal domain and so $W\cong W^{|L|}$, the $\fp$-adic metric for
$\fp\in S$ is ``twisted'' in general. Therefore the subtleness
remains also in this case.
\end{rem}
We can projectivise the whole construction as follows. We put $\overline{\cF}:=\cF\oplus\cO\cdot z$
with $z$ a dummy variable and extend the $G$-action on $\cF$ to
$\overline{\cF}$ by the trivial action on $\cO\cdot z$. Applying
the construction of $\Xi_{L}$ and $\Theta_{L}$ to $\overline{\cF}$,
we obtain $\overline{\Xi_{L}}=\Xi_{L}\oplus\cO$ and $\overline{\Theta_{L}}=\Theta_{L}\oplus\cO\cdot z$.
We let 
\[
\overline{W}:=\Proj S^{\bullet}\overline{\cF}\text{ and }\overline{W}^{|L|}:=\Proj S^{\bullet}\overline{\Theta_{L}}.
\]
 Let $\overline{W}^{\left\langle L\right\rangle }:=W^{|L|}\otimes_{\cO}\cO_{L}$.
We obtain the commutative diagram: 
\[
\xymatrix{ & \overline{W}^{\left\langle L\right\rangle }\ar[dl]\ar[dr]\\
\overline{W}\ar[dr] &  & \overline{W}^{|L|}\ar[dl]\\
 & \overline{W}/G
}
\]
Each entry of the diagram contains the corresponding entry in the
preceding one (\ref{eq:untw diag}) as an open dense subscheme.

\subsection{Over a number field}

We apply the untwisting to the case where $T=\Spec K$ with $K$ the
given number field and $V=W=\AA_{K}^{d}$. Let $X:=V/G$ as above.
For each $L\in G\Tor(K)$, we have a one-to-one correspondence
\[
V(L)^{G}/\Aut(L)\leftrightarrow V^{|L|}(K)/\Aut(L).
\]
Let $V_{\ur}^{|L|}$ be the preimage of $X_{\ur}$. We obtain bijections
\[
\bigsqcup_{L\in G\Tor(K)}V_{\ur}^{|L|}(K)/\Aut(L)\to X_{\ur}(K)
\]
and
\begin{equation}
\bigsqcup_{L\in G\Fie(K)}V_{\ur}^{|L|}(K)/\Aut(L)\to X_{\prim}(K).\label{eq:untw-prim}
\end{equation}

Let us suppose that $\overline{V}\to\overline{X}$ is étale in codimension
one. Then so is $\overline{V}^{|L|}\to\overline{X}$, which we denote
by $\pi_{L}$. We have $K_{\overline{V}^{|L|}}=\pi_{L}^{*}K_{\overline{X}}$.
We give an adelic metric to $-K_{\overline{X}}$ and the induced one
to $-K_{\overline{V}^{|L|}}$. We consider the associated heights
on $\overline{X}$ and $\overline{V}^{|L|}$ and the height zeta functions
$Z_{X_{\prim}(K)}(s)$ and $Z_{V_{\ur}^{|L|}(K)}(s)$ with respect
to these heights. From (\ref{eq:ht-eq}), the height function on $\overline{X}(K)$
is stable under the $\Aut(L)$-action and the map (\ref{eq:untw-prim})
preserves heights. Since the $\Aut(L)$-action on $V_{\ur}^{|L|}(K)$
is free for every $L\in G\Fie(K)$ (Lemma \ref{lem:free-action}),
we obtain the following consequence:
\begin{prop}
We have
\begin{equation}
Z_{X_{\prim}(K)}(s)=\frac{1}{\sharp Z(G)}\sum_{L\in G\Fie(K)}Z_{V_{\ur}^{|L|}(K)}(s).\label{eq:ht-zetas}
\end{equation}

\end{prop}

\subsection{The distribution of $K$-points of $V_{\protect\ur}^{|L|}$\label{sub:distribution-VL}}

Let us now choose a finite set of places, $S\subset M_{K}$, with
$M_{K,\infty}\subset S$ such that the $G$-action on $V$ extends
to $\cV=\AA_{\cO_{S}}^{d}$. Let $\overline{\cX}:=\overline{\cV}/G$.
We suppose that the metrized $-K_{\overline{X}}$ has the model $-K_{\overline{\cX}/\cO_{S}}$.
Then $-K_{\overline{V}^{|L|}}$ with the induced metric has the model
\[
-\pi_{L}^{*}K_{\overline{\cX}/\cO_{S}}=-K_{\overline{\cV}^{|L|}/\cO_{S}}+K_{\overline{\cV}^{|L|}/\overline{\cX}}.
\]

Let us see how the relative canonical divisor $K_{\overline{\cV}^{|L|}/\overline{\cX}}$
is determined. Firstly, since $\overline{V}^{|L|}\to\overline{V}$
is étale in codimension one, $K_{\overline{\cV}^{|L|}/\overline{\cX}}$
is supported on the union of finitely many closed fibers of $\overline{\cV}^{|L|}\to\Spec\cO_{S}$.
For $\fp\in S^{c}$, let $F_{\fp}$ be the fiber of $\cV^{|L|}\to\Spec\cO_{S}$
over $\fp$ and $\overline{F_{\fp}}$ its closure in $\overline{\cV}^{|L|}$.
If we write 
\[
K_{\overline{\cV}^{|L|}/\overline{\cX}}=\sum_{\fp\in S^{c}}a_{\fp}\cdot\overline{F_{\fp}},
\]
then 
\[
K_{\cV^{|L|}/\cX}=\sum_{\fp\in S^{c}}a_{\fp}\cdot F_{\fp}.
\]
To determine the coefficients $a_{\fp}$, let $L_{\fp}=L\otimes_{K}K_{\fp}$
be the $G$-ring over $K_{\fp}$ induced from $L$. Since the untwisting
is compatible with the base change from $\cO_{S}$ to $\cO_{\fp}$,
we have 
\[
K_{(\cV_{\cO_{\fp}})^{|L_{\fp}|}/\cX_{\cO_{\fp}}}=a_{\fp}\cdot F_{\fp},
\]
which was computed in \cite{Yasuda:2014fk2}. Let us write $\cV_{\cO_{\fp}}=\Spec S^{\bullet}\cF$
for a free $\cO_{\fp}$-module $\cF$. Then $\Xi_{L,\fp}:=\Xi_{L}\otimes_{\cO_{S}}\cO_{\fp}$
is identical to 
\[
\Hom_{\cO_{\fp}}^{G}(\cF,\cO_{L_{\fp}}),
\]
which is an $\cO_{\fp}$-submodule of the free $\cO_{L_{\fp}}$-module
$\Hom_{\cO_{\fp}}(\cF,\cO_{L_{\fp}})$. 
\begin{prop}[{\cite[Lem. 6.5]{Yasuda:2014fk2}}]
We have
\[
a_{\fp}=\frac{1}{\sharp G}\cdot\length\frac{\Hom_{\cO}(\cF,\cO_{L_{\fp}})}{\cO_{L_{\fp}}\cdot\Xi_{L,\fp}}.
\]

\end{prop}
For $L\in G\Fie(K)$, let $C_{\Peyre,L}$ be Peyre's constant from
Theorem \ref{thm:Peyre} for the projective space $\overline{V}^{|L|}$
and the induced metric on its anti-canonical divisor. For any Zariski
open dense subset $U\subset\overline{V}^{|L|}$, 
\[
N_{U(K)}(B)\sim C_{\Peyre,L}B.
\]

\subsection{$V$-discriminants\label{sub:V-discriminants}}

We now introduce invariants of a $G$-ring $L$ over $K$ which measures
its ramification as well as the difference of the measures on $\overline{V}$
and $\overline{V}^{|L|}$.
\begin{defn}
\label{def:V-disc}For $L\in G\Fie(K)$, we define the $V$-\emph{discriminant
$D_{L}^{V}$ }and the \emph{extended $V$-discriminant $\tilde{D}_{L}^{V}$}
of $L\in G\Fie(K)$ as
\begin{gather*}
D_{L}^{V}:=\prod_{\fp\in S^{c}}\frac{\mu_{\fp}(V(K_{\fp}))}{\mu_{\fp}^{|L|}(\overline{V}^{|L|}(K_{\fp}))}=\prod_{\fp\in S^{c}}N_{\fp}^{a_{\fp}},\\
\tilde{D}_{L}^{V}:=\prod_{\fp\in M_{K}}\frac{\mu_{\fp}(V(K_{\fp}))}{\mu_{\fp}^{|L|}(\overline{V}^{|L|}(K_{\fp}))}=\prod_{\fp\in S^{c}}N_{\fp}^{a_{\fp}}\times\prod_{\fp\in S}\frac{\mu_{\fp}(V(K_{\fp}))}{\mu_{\fp}^{|L|}(\overline{V}^{|L|}(K_{\fp}))}.
\end{gather*}
We often omit the superscript $V$ if it causes no confusion.
\end{defn}
The (extended) $V$-discriminant is a global version of $N_{\fp}^{a_{\fp}}$,
which was studied by Wood and the author \cite{Wood-Yasuda-I} and
whose motivic relative was previously considered by the author \cite{Yasuda:2013fk}.
Dummit \cite{Dummit:2014vb} introduced the notion of $\rho$-\emph{discriminants,
}also using the tuning module. It seems that $V$-discriminants and
$\rho$-discriminants are closely related and differ only by some
constant factor, although the author does not know how to relate them
precisely. 

There exists constants $C_{1},C_{2}>0$ independent of $L$ such that
\begin{equation}
C_{1}D_{L}\le\tilde{D}_{L}\le C_{2}D_{L}.\label{eq:V-disc-diff}
\end{equation}
This is true because the $\fp$-factor of $\tilde{D}_{L}$ depends
only on the $G$-ring $L\otimes_{K}K_{\fp}$ over $K_{\fp}$. For
each $\fp$, there exist only finitely many such $G$-rings up to
isomorphism. Therefore 
\[
\frac{\tilde{D}_{L}}{D_{L}}=\prod_{\fp\in S}\frac{\mu_{\fp}(V(K_{\fp}))}{\mu_{\fp}^{|L|}(\overline{V}^{|L|}(K_{\fp}))}
\]
can take only finitely many distinct values. Thus, which we consider
$D_{L}$ or $\tilde{D}_{L}$ does not cause much difference, when
concerning the distribution of $G$-fields, which is discussed in
later sections. 

Let $C_{\Peyre,0}$ be Peyre's constant for the projective space $V$
and the given metric on its anti-canonical divisor, and $C_{\Peyre,L}$
the one for $V^{|L|}$ as above. The extended $V$-discriminant is
characterized by 
\begin{equation}
C_{\Peyre,L}=\tilde{D}_{L}^{-1}C_{\Peyre,0}.\label{eq:Cdiff}
\end{equation}

The following result justifies the statement that $V$-discriminants
measure ramification.
\begin{lem}
Suppose that the $G$-action on $\cV_{\kappa_{\fp}}$ is faithful.
For $\fp\in S^{c}$, $a_{\fp}>0$ if and only if $L/K$ is ramified
at $\fp$. \end{lem}
\begin{proof}
To compute $a_{\fp}$, we can consider the untwisting with the base
$\cO_{\fp}$. Furthermore, from \cite{Wood-Yasuda-I}, we can replace
the base with the completion of the maximal unramified extension of
$\cO_{\fp}$; we denote it by $\cO_{\fp}^{\ur}$ and its function
field by $K_{\fp}^{\ur}$. We denote the $G$-ring $L\otimes_{K}K_{\fp}^{\ur}$
over $K_{\fp}^{\ur}$ by $L_{\fp}^{\ur}$. We get the following diagram.
\[
\xymatrix{ & \cV_{\cO_{\fp}^{\ur}}^{\left\langle L_{\fp}^{\ur}\right\rangle }\ar[dl]\ar[dr]\\
\cV_{\cO_{\fp}^{\ur}} &  & \cV_{\cO_{\fp}^{\ur}}^{|L_{\fp}^{\ur}|}
}
\]
Let $H$ be the stabilizer of a connected component of $\Spec L_{\fp}^{\ur}$
and $\left(\cV_{\cO_{\fp}^{\ur}}^{\left\langle L_{\fp}^{\ur}\right\rangle }\right)_{0}$
the corresponding connected component of $\cV_{\cO_{\fp}^{\ur}}^{\left\langle L_{\fp}^{\ur}\right\rangle }$.
We note that $H=1$ if and only if $L/K$ is unramified at $\fp$.
From \cite[Lem. 5.7]{Yasuda:2014fk2}, $\sharp H\cdot a_{\fp}$ is
the multiplicity of the Jacobian ideal of the morphism
\begin{equation}
\left(\cV_{\cO_{\fp}^{\ur}}^{\left\langle L_{\fp}^{\ur}\right\rangle }\right)_{0}\to\cV_{\cO_{\fp}^{\ur}}\otimes_{\cO_{\fp}^{\ur}}\cO_{L_{\fp}^{\ur}}\label{eq:mor-jac}
\end{equation}
along the closed fiber
\[
\left(\cV_{\cO_{\fp}^{\ur}}^{\left\langle L_{\fp}^{\ur}\right\rangle }\right)_{0}\otimes_{\tilde{\cO}_{\fp}^{\ur}}\overline{\kappa_{\fp}}.
\]
If $H\ne1$, then the closed fiber maps onto the $H$-fixed point
locus in $\cV_{\overline{\kappa_{\fp}}}$, which is by assumption
a proper subset of the closed fiber. Therefore the Jacobian ideal
has positive multiplicity and hence $a_{\fp}>0$. If $H=1$, then
(\ref{eq:mor-jac}) is an isomorphism and the Jacobian ideal is trivial.
Hence $a_{\fp}=0$.\end{proof}
\begin{prop}
\label{prop:fin many torsors}For every real number $B>0$, there
exist at most finitely many $L\in G\Tor(K)$ with $D_{L}\le B$. Similarly
for $\tilde{D}_{L}$. \end{prop}
\begin{proof}
From (\ref{eq:V-disc-diff}), it suffices to consider the case of
$D_{L}$. Replacing $S$ if necessary, we may suppose that for every
$\fp\in S^{c}$, the $G$-action on $\cV_{\kappa_{\fp}}$ is faithful.
If $L/K$ is ramified at $\fp\in S^{c}$, then the lemma above shows
$D_{L}\ge N_{\fp}^{1/\sharp G}$. For the inequality $D_{L}\le B$
holding, it is necessary that $L/K$ is unramified over those $\fp\in S^{c}$
with $N_{\fp}^{1/\sharp G}>B$, in particular, over almost all $\fp\in S$.
From \cite[p.\ 203]{MR1697859}, there exist only finitely many such
$G$-rings $L$. 
\end{proof}

\section{Modified heights\label{sec:Modified-heights}}

From (\ref{eq:Cdiff}) one might naively expect that heights of $K$-points
of $V^{|L|}$ is roughly $\tilde{D}_{L}^{-1}$ times the ones of the
corresponding points of $V$ through a certain isomorphism $V\cong V^{|L|}$.
Since there exists no canonical isomorphism between $V$ and $V^{|L|}$,
and even if we make good choices of isomorphisms, since the changes
of $\fp$-adic metrics for $\fp\in S$ seem to be hard to trace (see
Remark \ref{rem:tw-untw}), it must be difficult to know a precise
relation between heights on $V$ and $V^{|L|}$. However, if we assume
the naive expectation to be true, then we can explain Conjecture \ref{conj:BT}
in terms of the distribution of $G$-rings over $K$ and vice versa.

We keep the notation from Sections \ref{sub:distribution-VL} and
\ref{sub:V-discriminants}. In particular, we fix a metric on $-K_{\overline{X}}$
given by a model $-K_{\overline{\cX}/\cO_{S}}$ and the induced height
function $H$ on $\overline{X}(K)$. We also fix the induced height
functions on $\overline{V}(K)$ and $\overline{V}^{|L|}(K)$, which
we denote again by $H$. For each $L\in G\Fie(K)$, we fix a $K$-linear
isomorphism 
\begin{equation}
\psi_{L}:V\xrightarrow{\sim}V^{|L|}.\label{eq:chosen isom}
\end{equation}
 
\begin{defn}
We define the \emph{modified height }of $y\in V^{|L|}(K)$ as
\[
H^{\modified}(y)=\frac{H(\psi_{L}^{-1}(y))}{\tilde{D}_{L}}.
\]
We define the \emph{modified height zeta function $Z_{V^{|L|}(K)}^{\modified}(s)$
}of\emph{ $V^{|L|}(K)$ }by
\[
\sum_{x\in V^{|L|}(K)}H^{\modified}(x)^{-s}.
\]

\end{defn}
By definition, 
\[
Z_{V^{|L|}(K)}^{\modified}(s)=\tilde{D}_{L}^{-s}\cdot Z_{V(K)}(s).
\]
Let 
\[
Z_{V^{|L|}(K)}(s)=\sum_{x\in V^{|L|}(K)}H(x)^{-s}
\]
be the height zeta function of $V^{|L|}(K)$ with respect to the given
(not modified) height. If we assume meromorphic continuation of these
functions as in Condition \ref{cond:merom-continuation}, then they
have the right-most pole at the same place, of the same order (order
one) and of the same residue. Therefore we think of $Z_{V^{|L|}(K)}^{\modified}(s)$
as an approximation of $Z_{V^{|L|}(K)}(s)$. 
\begin{defn}
We define the \emph{extended $V$-discriminant zeta function of $G$-fields
}by\emph{
\[
Z_{V,G,K}^{\disc}(s)=Z^{\disc}(s):=\frac{1}{\sharp Z(G)}\cdot\sum_{L\in G\Fie(K)}\tilde{D}_{L}^{-s}.
\]
}
\end{defn}
We then think of the function, 
\begin{equation}
\cZ(s):=\frac{1}{\sharp Z(G)}\cdot\sum_{L\in G\Fie(K)}Z_{V^{|L|}(K)}^{\modified}(s)=Z^{\disc}(s)Z_{V(K)}(s),\label{eq:Main fn eq}
\end{equation}
as an approximation of $Z_{X_{\prim}(K)}(s)$, and heuristically expect
that $\cZ(s)$ and $Z_{X_{\prim}(K)}(s)$ have the right-most poles
at the same place and of the same order.

Let us suppose that the functions $Z_{V(K)}(s)$ and $Z^{\disc}(s)$
admits meromorphic continuation as in Condition \ref{cond:merom-continuation}.
Suppose that the right-most pole of $Z^{\disc}(s)$ is at $s=\alpha$
and has order $\beta$. The place and the order of the right-most
pole of $\cZ(s)$ are then given in Table \ref{table}.

\begin{table}
\begin{tabular}{|c|c|c|}
\hline 
 & place & order\tabularnewline
\hline 
\hline 
$\alpha<1$ & 1 & 1\tabularnewline
\hline 
$\alpha=1$ & 1 & $\beta+1$\tabularnewline
\hline 
$\alpha>1$ & $\alpha$ & $\beta$\tabularnewline
\hline 
\end{tabular}

\protect\caption{the right-most pole of $\protect\cZ(s)$}
\label{table}
\end{table}

\section{Distribution of number fields\label{sec:Distribution-of-number}}

In this section we first recall Malle's conjecture \cite{MR1884706,MR2068887}
on the distribution of extensions of a number field. We then propose
a generalization of it, which we will relate with a version of Manin's
conjecture, Conjecture \ref{conj:BT}, in the next section.

\subsection{Malle's conjecture}

Let $G$ be a transitive finite subgroup of the symmetric group $S_{n}$,
acting on $[n]:=\{1,\dots,n\}$. 
\begin{defn}
A degree $n$ field extension $L$ of $K$ is called a \emph{small
$G$-field (over $K$) }if its Galois closure $\hat{L}/K$ has a Galois
group which is permutation isomorphic to $G$. Two small $G$-fields
are said to be \emph{isomorphic }if there is a $K$-isomorphism between
them. The set of small $G$-fields over $K$ up to isomorphism is
denoted by $G\fie(K)$. 
\end{defn}
Small $G$-fields and large $G$-fields are related by the map
\[
G\Fie(K)\to G\fie(K),\,L\mapsto L^{G_{1}},
\]
where $G_{1}\subset G$ is the stabilizer subgroup of $1\in[n]$. 
\begin{lem}
\label{lem:map-fiber}The map above is a $\frac{\sharp N_{S_{n}}(G)\sharp Z(G)}{\sharp C_{S_{n}}(G)\sharp G}$-to-one
surjection.\end{lem}
\begin{proof}
The map is clearly surjective. Let $S_{G,K}$ be the set of continuous
surjections of $\Gal(\bar{K}/K)$ to $G$. The natural map
\[
S_{G,K}\to G\Fie(K)
\]
can be identified with the quotient map associated to the $G$-action
on $S_{G,K}$ by conjugation. Therefore this map is a $\frac{\sharp G}{\sharp Z(G)}$-to-one
surjection. On the other hand, the natural map 
\begin{equation}
S_{G,K}\to G\fie(K)\label{eq:G-fie-1}
\end{equation}
is identified with the restriction of the natural map
\[
T_{S_{n},K}\to n\fie(K),
\]
where $T_{S_{n},K}$ is the set of continuous homomorphisms $\Gal(\overline{K}/K)\to S_{n}$
whose images are transitive subgroups and $n\fie(K)$ is the set of
isomorphism classes of degree $n$ field extensions of $K$. The last
map is, in turn, identified with the quotient map associated to the
$S_{n}$-conjugation. Hence map (\ref{eq:G-fie-1}) is a $\frac{\sharp N_{S_{n}}(G)}{\sharp C_{S_{n}}(G)}$-to-one
surjection. We have proved the lemma.\end{proof}
\begin{notation}
For a finite extension $L/K$, let $\mathrm{Disc}_{L/K}$ be its discriminant
and let $d_{L}:=|N_{K/\QQ}(\mathrm{Disc}_{L/K})|$. For a real number
$B>0$, we consider the number of small $G$-fields with $d_{L}$
bounded: 
\[
n(G,B):=\sharp\{L\in G\fie(K)\mid d_{L}\le B\}.
\]

\end{notation}
Malle \cite{MR1884706,MR2068887} raised a conjecture concerning the
asymptotic behavior of $n(G,B)$ as $B$ tends to the infinity. To
state it, we define some invariants. 
\begin{defn}
For $g\in G$, we define its \emph{index}, denoted $\ind(g)$, to
be 
\[
n-\sharp\{g\text{-orbits in }[n]\}.
\]
We define the \emph{index }of $G$, $\ind(G)$, to be the minimum
of $\ind(g)$, $g\in G\setminus\{1\}$. 
\end{defn}
Since the index of an element of $G$ depends only on its $K$-conjugacy
class, the index of a $K$-conjugacy class makes sense. Let $\beta(G)$
denote the number of $K$-conjugacy classes having index equal to
$\ind(G)$. 
\begin{conjecture}[Malle's conjecture \cite{MR1884706,MR2068887}]
\label{conj:Malle}If $K$ is sufficiently large, then we have 
\[
n(G,B)\sim CB^{1/\ind(G)}(\log B)^{\beta(G)-1}.
\]

\end{conjecture}
In fact, the condition ``if $K$ is sufficiently large'' was not
assumed in Malle's original formulation. That version, which is stronger
than ours, holds for abelian groups \cite{MR791087,MR969545}. For
non-abelian groups, we only refer the reader to surveys \cite{Belabas:2005tc,MR1957027,MR2330428}.
However there exist counter-examples for the original version  \cite{MR2135320}. 
\begin{rem}
Klüners' counter-examples \cite{MR2135320} are based on the following
phenomenon: in some case, there exist intermediate extensions $K(\zeta)/K$
such that $\zeta$ is a root of unity and the value of $\beta$ differs
for $K$ and $K(\zeta)$. If we suppose that $K$ is sufficiently
large as in our version of the conjecture, then this problem does
not occur. Another way to avoid the problem is to exclude extensions
containing such an intermediate extension $K(\zeta)$ from the counting.
Klüners already discussed it, but called it ``not very natural''
(the last section, \emph{loc.cit.}). However, according to the correspondences
(\ref{eq:pt-corr}) and (\ref{eq:untw-prim}), such an exclusion corresponds
to removing an additional accumulating thin subset from $X_{\prim}(K)$
(see the following lemma) and getting a smaller cothin subset. It
might be thus natural from the viewpoint of the ``cothin version''
of Manin's conjecture. \end{rem}
\begin{lem}
Let $V$ be a quasi-projective geometrically irreducible $F$-variety
with a faithful $G$-action and consider bijection (\ref{eq:pt-corr}).
Let $E/F$ be a non-trivial finite field extension and $G\Fie(F)_{E}$
the subset of $G\Fie(F)$ consisting of those $G$-fields $L$ such
that there exists a $K$-embedding $E\hookrightarrow L$. The image
of 
\[
\bigsqcup_{L\in G\Fie(F)_{E}}V(L)^{G}/\Aut(L)
\]
in $X(F)$ is thin.\end{lem}
\begin{proof}
Replacing it with the Galois closure, we may suppose that $E/F$ is
Galois. Let $H$ denote its Galois group. For $L\in G\Fie(F)_{E}$,
an embedding $E\hookrightarrow L$ gives a surjective homomorphism
$\psi:G\twoheadrightarrow H$ and we have a natural map
\[
V(L)^{G}\to(V/\Ker(\psi))(E)^{H},
\]
which factors the natural map $V(L)^{G}\to X(K)$. Therefore, the
image of the problem is contained in the image of 
\[
\bigsqcup_{\psi:G\twoheadrightarrow H}(V/\Ker(\psi))(E)^{H}.
\]
It is now enough to show the following claim:
\begin{claim*}
Let $V$ be as in the lemma and $L\in G\Fie(F)$. Suppose $G\ne1$.
Then the image of $V(L)^{G}\to X(F)$ is thin.
\end{claim*}
To show this, we may suppose that $V$ is an affine variety, and can
use the untwisting technique for an arbitrary affine variety in \cite{Yasuda:2014fk2}.
We can construct the untwisting variety $V^{|L|}$ of $V$ as follows.
Firstly there exists a $G$-equivariant closed embedding of $V$ into
an affine space $W=\AA_{F}^{d}$ with a linear faithful $G$-action.
Then $V^{|L|}$ is defined as the preimage of $X=V/G\subset W/G$
in $W^{|L|}$. Note that since we are working in characteristic zero,
we do not need to take the normalization as did in \cite{Yasuda:2014fk2}.
There exists a natural generically finite morphism $V^{|L|}\to X$
of degree $\sharp G$ such that the image of $V_{\ur}(L)^{G}$ in
$X(F)$ is contained in the image of $V^{|L|}(F)$. The construction
shows $V^{|L|}\otimes_{F}L\cong_{L}V\otimes_{F}L.$ It follows that
$V^{|L|}$ is irreducible. The claim, and hence the lemma follow. 
\end{proof}

\subsection{$V$-discriminants vs.\ ordinary discriminants}

Local counterparts of $V$-discriminants naturally appear as counting
weights of $G$-rings over a local field, especially in the context
of the McKay correspondence \cite{Wood-Yasuda-I,MR3230848,wild-p-adic,Yasuda:2013fk,Yasuda:2014fk2}.
It is natural to use (extended) $V$-discriminats when counting $G$-rings/fields
over a number field as well, and relate them to relevant quotient
varieties. Recently Dummit \cite{Dummit:2014vb} studied the distribution
of $G$-fields with respect to his $\rho$-discriminant, which is
seemingly close (essentially the same) to our $V$-discriminant. 

The relation between $V$-discriminants and ordinary ones is as follows.
Let $G$ be a transitive subgroup of $S_{n}$ and let $V=\AA_{K}^{2n}$
be the direct sum of two copies of the associated $n$-dimensional
permutation $G$-representation. Putting $S=M_{K,\infty}$, we define
$\fp$-adic metrics on $-K_{\overline{V}}$ given by the model $-K_{\overline{\cV}/\cO_{K}}$.
Accordingly we associate the $V$-discriminant $D_{L}^{V}$ to each
$L\in G\Fie(K)$. 
\begin{prop}
\label{prop:compare-discs}We have $D_{L}^{V}=d_{L^{G_{1}}}$.\end{prop}
\begin{proof}
The corresponding assertion over a local field was proved in \cite{Wood-Yasuda-I}.
The proposition is a direct consequence of it.
\end{proof}
For $g\in G$, let us define its index by regarding $G$ as a subgroup
of $S_{n}$ and its age by regarding $G$ as a subgroup of $\GL(V)$
with $V$ as above. 
\begin{prop}
\label{prop:ind=00003Dage}We have $\ind(g)=\age(g)$.\end{prop}
\begin{proof}
From the additivity of age and index, we may suppose that $g$ is
the cyclic permutation, $1\mapsto2\mapsto\cdots\mapsto n\mapsto1$.
Then a diagonalization of $g$ as an element of $\GL_{2n}(\overline{K})$
is
\[
\diag(1,\zeta,\dots,\zeta^{n-1})^{\oplus2}
\]
with $\zeta$ a primitive $n$-th root of unity. We have
\[
\age(g)=\frac{2}{n}\sum_{i=0}^{n-1}i=n-1=\ind(g).
\]

\end{proof}

\subsection{General actions}

Next we consider an arbitrary faithful linear $G$-action on $V=\AA_{K}^{d}$
such that the associated map $\overline{V}\to\overline{X}$ is étale
in codimension one. We fix an adelic metric on the anti-canonical
divisor of $\overline{X}=\overline{V}/G$ and the one induced on the
anti-canonical divisor of $\overline{V}$. We also fix $S$ such that
the action extends to $\AA_{\cO_{S}}^{d}$. In this setting, we define
(extended) $V$-discriminants for $L\in G\Fie(K)$. For $B>0$, we
put
\begin{gather*}
\tilde{N}(G,V,B):=\sharp\{L\in G\Fie(K)\mid\tilde{D}_{L}\le B\}\\
N(G,V,B):=\sharp\{L\in G\Fie(K)\mid D_{L}\le B\},
\end{gather*}
which are finite thanks to Proposition \ref{prop:fin many torsors}.
We would like to propose the following generalization of Malle's conjecture:
\begin{conjecture}
\label{conj:general Malle}If $K$ is sufficiently large, we have
\[
N(G,V,B)\sim CB^{1/\age(G)}(\log B)^{\upsilon(G)-1},
\]
with $\upsilon(G)$ the number of the youngest $K$-conjugacy classes.
\end{conjecture}
Thanks to Propositions \ref{prop:compare-discs} and \ref{prop:ind=00003Dage},
Conjecture \ref{conj:general Malle} is indeed a generalization of
Conjecture \ref{conj:Malle}. From inequalities (\ref{eq:V-disc-diff}),
we expect that the same asymptotic formula holds for $\tilde{N}(G,V,B)$,
the constant factor $C$ possibly changing. 
\begin{example}
Let $G\subset S_{n}$ be a transitive subgroup and consider the natural
$G$-action on $V=\AA_{K}^{d}$. For an integer $n>0$, let $V^{\oplus n}$
denote the direct sum $G$-representation of $n$ copies of $V$.
We have
\[
D_{L}^{V^{\oplus n}}=(D_{L}^{V})^{n}.
\]
Therefore Conjecture \ref{conj:Malle} for $G$ implies Conjecture
\ref{conj:general Malle} for $V^{\oplus n}$. 
\end{example}

\begin{example}
Suppose that $G=\left\langle g\right\rangle $ is a cyclic group of
order $p$ with $p$ an odd prime number, that $K$ contains a primitive
$p$-th root $\zeta$ of unity and that $V=\AA_{K}^{2n}$ has the
$G$-action determined by
\[
g=\diag(\zeta^{a_{1}},\zeta^{-a_{1}},\zeta^{a_{2}},\zeta^{-a_{2}},\dots,\zeta^{a_{n}},\zeta^{-a_{n}})\in\GL(V).
\]
The $G$-representation $V$ is balanced in the sense of \cite{Wood-Yasuda-I}.
Moreover the function 
\[
G\to\QQ,\,g\mapsto\age(g)
\]
is fair in the sense of \cite{Wood:2010gs}. From \cite[Th. 3.1]{Wood:2010gs},
Conjecture \ref{conj:general Malle} is true in this case.\end{example}
\begin{rem}
\label{rem:fields-than-rings}The reason why we consider the distribution
of $G$-fields rather than $G$-rings can be explained as follows.
Let us consider the case $G=S_{n}$. There exists a one-to-one correspondence
between $S_{n}$-rings over $K$ and degree $n$ $K$-algebras. In
general, $K$-algebras decomposable as the product of several fields
outnumber indecomposable ones. In the context of Manin's conjecture,
this corresponds to the expected phenomenon that non-primitive rational
points of a quotient variety $X=V/G$ outnumber primitive ones. Thus
non-primitive rational points constitute an \emph{accumulating thin
subset, }which should be removed to have a nice formula.
\end{rem}

\section{Manin vs.\  Malle\label{sec:Manin-vs.-Malle}}

In this last section, we see how Conjectures \ref{conj:BT} and \ref{conj:general Malle}
are related as a consequence of materials prepared in earlier sections. 

We suppose that $V=\AA_{K}^{d}$ is given a faithful linear $G$-action
and that $\overline{V}\to\overline{X}$ is étale in codimension one.
Let $S\subset M_{K}$ be a finite subset containing $M_{K,\infty}$
such that the $G$-action extends to $\AA_{\cO_{S}}^{d}$. Accordingly
we fix an adelic metric on $-K_{\overline{X}}$ and the induced one
on $-K_{\overline{V}}$. We assume that the Dirichlet series $Z^{\disc}(s)$,
$Z_{V(K)}(s)$ and $Z_{X_{\prim}(K)}(s)$ admit meromorphic continuation
as in Condition \ref{cond:merom-continuation}. We also assume that
$\cZ(s)=Z^{\disc}(s)Z_{V(K)}(s)$ and $Z_{X_{\prim}(K)}(s)$ have
the right-most poles at the same place and of the same order, and
assume that $K$ is sufficiently large. 

We discuss three cases separately.

\subsection{The case $\protect\age(G)>1$}

In this case, Conjecture \ref{conj:general Malle} shows that the
Dirichlet series $Z^{\disc}(s)$ is convergent at $s=1$. This and
the assumptions show that $\cZ(s)$ and hence $Z_{X_{\prim}(K)}(s)$
have simple poles at $s=1$. Thus Conjecture \ref{conj:BT} holds
for $U=X_{\prim}(K)$. 

Conversely, if Conjecture \ref{conj:BT} holds for $U=X_{\prim}(K)$,
then $Z_{X_{\prim}(K)}(s)$ and $\cZ(s)$ have the right-most poles
at $s=1$, which are simple poles. Therefore $Z^{\disc}(s)$ is convergent
at $s=1$ and the right-most pole of $Z^{\disc}(s)$ is in $s>1$.
We thus have only a weaker estimation than Conjecture \ref{conj:general Malle}.

\subsection{The case $\protect\age(G)=1$}

This case is where Conjectures \ref{conj:BT} and \ref{conj:general Malle}
are tied to each other the best. Conjecture \ref{conj:general Malle}
says that $Z^{\disc}(s)$ has the right-most pole at $s=1$, which
has order $\upsilon(G)=\gamma(X)$ (Proposition \ref{prop:div min}).
Our assumptions show that $\cZ(s)$ and hence $Z_{X_{\prim}(K)}(s)$
have right-most poles at $s=1$, which have order $\gamma(X)+1$.
Thus Conjecture \ref{conj:general Malle} implies Conjecture \ref{conj:BT}
for $U=X_{\prim}(K)$ under the assumptions. We can similarly see
the converse.

\subsection{The case $\protect\age(G)<1$}

In this case, Conjecture \ref{conj:general Malle} says that $Z^{\disc}(s)$
has the abscissa of convergence at $s=1/\age(G)>1$. This shows Conjectures
\ref{conj:BT}. However, the expected abscissa of $Z_{X_{\prim}(B)}(s)$
(if $X_{\prim}(B)$ is sufficiently small) is not at $s=1/\age(G)$
at least when $G$ is abelian (see Corollary \ref{cor:toric-strict-ineq}). 

If Conjecture \ref{conj:BT} is true, then the abscissa of $Z^{\disc}(s)$
is greater than one. 

 \bibliographystyle{plain}
\bibliography{/Users/Takehiko/Dropbox/Math_Articles/mybib}

\begin{thebibliography}{10}

\bibitem{MR1423638}
V.~Batyrev and Yu. Tschinkel.
\newblock Height zeta functions of toric varieties.
\newblock {\em J. Math. Sci.}, 82(1):3220--3239, 1996.
\newblock Algebraic geometry, 5.

\bibitem{MR1032922}
V.~V. Batyrev and Yu.~I. Manin.
\newblock Sur le nombre des points rationnels de hauteur born{\'e} des
  vari{\'e}t{\'e}s alg{\'e}briques.
\newblock {\em Math. Ann.}, 286(1-3):27--43, 1990.

\bibitem{MR1677693}
Victor~V. Batyrev.
\newblock Non-{A}rchimedean integrals and stringy {E}uler numbers of
  log-terminal pairs.
\newblock {\em J. Eur. Math. Soc. (JEMS)}, 1(1):5--33, 1999.

\bibitem{Batyrev:1996ut}
Victor~V. Batyrev and Yuri Tschinkel.
\newblock {Rational points on some Fano cubic bundles}.
\newblock {\em C. R. Acad. Sci. Paris S\'er. I Math.}, 323(1):41--46, 1996.

\bibitem{MR1620682}
Victor~V. Batyrev and Yuri Tschinkel.
\newblock Manin's conjecture for toric varieties.
\newblock {\em J. Algebraic Geom.}, 7(1):15--53, 1998.

\bibitem{MR1679843}
Victor~V. Batyrev and Yuri Tschinkel.
\newblock Tamagawa numbers of polarized algebraic varieties.
\newblock {\em Ast{\'e}risque}, (251):299--340, 1998.
\newblock Nombre et r{{\'e}}partition de points de hauteur born{{\'e}}e (Paris,
  1996).

\bibitem{Belabas:2005tc}
Karim Belabas.
\newblock {Param\'etrisation de structures alg\'ebriques et densit\'e de
  discriminants (d'apr\`es Bhargava)}.
\newblock {\em Ast\'erisque}, (299):Exp. No. 935, ix, 267--299, 2005.

\bibitem{MR2354798}
Manjul Bhargava.
\newblock Mass formulae for extensions of local fields, and conjectures on the
  density of number field discriminants.
\newblock {\em Int. Math. Res. Not. IMRN}, (17):Art. ID rnm052, 20, 2007.

\bibitem{Boucksom:2013er}
S{\'e}bastien Boucksom, Jean-Pierre Demailly, Mihai P{\u{a}}un, and Thomas
  Peternell.
\newblock {The pseudo-effective cone of a compact K\"ahler manifold and
  varieties of negative Kodaira dimension}.
\newblock {\em J. Algebraic Geom.}, 22(2):201--248, 2013.

\bibitem{Browning:2013aa}
T.D. Browning and D.~Loughran.
\newblock Varieties with too many rational points.
\newblock available at http://arxiv.org/abs/1311.5755, 11 2013.

\bibitem{MR2647601}
Antoine Chambert-Loir.
\newblock Lectures on height zeta functions: at the confluence of algebraic
  geometry, algebraic number theory, and analysis.
\newblock In {\em Algebraic and analytic aspects of zeta functions and
  {$L$}-functions}, volume~21 of {\em MSJ Mem.}, pages 17--49. Math. Soc.
  Japan, Tokyo, 2010.

\bibitem{Chen:2014fh}
Dawei Chen and Izzet Coskun.
\newblock {Extremal effective divisors on M$_{1,n}$}.
\newblock {\em Math. Ann.}, 359(3-4):891--908, 2014.

\bibitem{MR1957027}
H.~Cohen.
\newblock Constructing and counting number fields.
\newblock In {\em Proceedings of the {I}nternational {C}ongress of
  {M}athematicians, {V}ol. {II} ({B}eijing, 2002)}, pages 129--138. Higher Ed.
  Press, Beijing, 2002.

\bibitem{MR2330428}
Henri Cohen, Francisco Diaz~y Diaz, and Michel Olivier.
\newblock Counting discriminants of number fields.
\newblock {\em J. Th{\'e}or. Nombres Bordeaux}, 18(3):573--593, 2006.

\bibitem{MR1905024}
Jan Denef and Fran{\c{c}}ois Loeser.
\newblock Motivic integration, quotient singularities and the {M}c{K}ay
  correspondence.
\newblock {\em Compositio Math.}, 131(3):267--290, 2002.

\bibitem{Dummit:2014vb}
Evan~P. Dummit.
\newblock {\em Counting number field extensions of given degree, bounded
  discriminant, and specified Galois closure}.
\newblock PhD thesis, 2014.

\bibitem{Ellenberg:2005bn}
Jordan~S Ellenberg and Akshay Venkatesh.
\newblock {Counting extensions of function fields with bounded discriminant and
  specified Galois group}.
\newblock In {\em Geometric methods in algebra and number theory}, pages
  151--168. Birkh\"auser Boston, Boston, MA, 2005.

\bibitem{Ellenberg:2006js}
Jordan~S Ellenberg and Akshay Venkatesh.
\newblock {The number of extensions of a number field with fixed degree and
  bounded discriminant}.
\newblock {\em Ann. of Math. (2)}, 163(2):723--741, 2006.

\bibitem{Franke:1989go}
Jens Franke, Yuri~I Manin, and Yuri Tschinkel.
\newblock {Rational points of bounded height on Fano varieties}.
\newblock {\em Invent. math.}, 95(2):421--435, 1989.

\bibitem{MR1463181}
Yukari Ito and Miles Reid.
\newblock The {M}c{K}ay correspondence for finite subgroups of {${\rm
  SL}(3,\bold C)$}.
\newblock In {\em Higher-dimensional complex varieties ({T}rento, 1994)}, pages
  221--240. de Gruyter, Berlin, 1996.

\bibitem{MR2135320}
J{{\"u}}rgen Kl{{\"u}}ners.
\newblock A counterexample to {M}alle's conjecture on the asymptotics of
  discriminants.
\newblock {\em C. R. Math. Acad. Sci. Paris}, 340(6):411--414, 2005.

\bibitem{MR3057950}
J{{\'a}}nos Koll{{\'a}}r.
\newblock {\em Singularities of the minimal model program}, volume 200 of {\em
  Cambridge Tracts in Mathematics}.
\newblock Cambridge University Press, Cambridge, 2013.
\newblock With a collaboration of S{{\'a}}ndor Kov{{\'a}}cs.

\bibitem{MR2095471}
Robert Lazarsfeld.
\newblock {\em Positivity in algebraic geometry. {I}}, volume~48 of {\em
  Ergebnisse der Mathematik und ihrer Grenzgebiete. 3. Folge. A Series of
  Modern Surveys in Mathematics [Results in Mathematics and Related Areas. 3rd
  Series. A Series of Modern Surveys in Mathematics]}.
\newblock Springer-Verlag, Berlin, 2004.
\newblock Classical setting: line bundles and linear series.

\bibitem{MR791087}
Sirpa M{\"a}ki.
\newblock On the density of abelian number fields.
\newblock {\em Ann. Acad. Sci. Fenn. Ser. A I Math. Dissertationes}, (54):104,
  1985.

\bibitem{Maki:1993et}
Sirpa M{\"a}ki.
\newblock {The conductor density of abelian number fields}.
\newblock {\em J. Lond. Math. Soc. (2)}, 47(1):18--30, 1993.

\bibitem{MR1884706}
Gunter Malle.
\newblock On the distribution of {G}alois groups.
\newblock {\em J. Number Theory}, 92(2):315--329, 2002.

\bibitem{MR2068887}
Gunter Malle.
\newblock On the distribution of {G}alois groups. {II}.
\newblock {\em Experiment. Math.}, 13(2):129--135, 2004.

\bibitem{MR1697859}
J{{\"u}}rgen Neukirch.
\newblock {\em Algebraic number theory}, volume 322 of {\em Grundlehren der
  Mathematischen Wissenschaften [Fundamental Principles of Mathematical
  Sciences]}.
\newblock Springer-Verlag, Berlin, 1999.
\newblock Translated from the 1992 German original and with a note by Norbert
  Schappacher, With a foreword by G. Harder.

\bibitem{MR1340296}
Emmanuel Peyre.
\newblock Hauteurs et mesures de {T}amagawa sur les vari{\'e}t{\'e}s de {F}ano.
\newblock {\em Duke Math. J.}, 79(1):101--218, 1995.

\bibitem{MR2019019}
Emmanuel Peyre.
\newblock Points de hauteur born{\'e}e, topologie ad{\'e}lique et mesures de
  {T}amagawa.
\newblock {\em J. Th{\'e}or. Nombres Bordeaux}, 15(1):319--349, 2003.
\newblock Les XXII{{\`e}}mes Journ{{\'e}}es Arithmetiques (Lille, 2001).

\bibitem{LeRudulierLine}
C{\'e}cile~Le Rudulier.
\newblock Points alg{\'e}briques de hauteur born{\'e}e sur la droite
  projective.
\newblock available at http://perso.univ-rennes1.fr/cecile.lerudulier/.

\bibitem{LeRudulierSurface}
C{\'e}cile~Le Rudulier.
\newblock Points alg{\'e}briques de hauteur born{\'e}e sur une surface.
\newblock available at http://perso.univ-rennes1.fr/cecile.lerudulier/.

\bibitem{MR0162787}
S.~Schanuel.
\newblock On heights in number fields.
\newblock {\em Bull. Amer. Math. Soc.}, 70:262--263, 1964.

\bibitem{MR2363329}
Jean-Pierre Serre.
\newblock {\em Topics in {G}alois theory}, volume~1 of {\em Research Notes in
  Mathematics}.
\newblock A K Peters, Ltd., Wellesley, MA, second edition, 2008.
\newblock With notes by Henri Darmon.

\bibitem{Wood:2010gs}
Melanie~Matchett Wood.
\newblock {On the probabilities of local behaviors in abelian field
  extensions}.
\newblock {\em Compos. Math.}, 146(1):102--128, 2010.

\bibitem{Wood-Yasuda-I}
Melanie~Matchett Wood and Takehiko Yasuda.
\newblock Mass formulas for local {G}alois representations and quotient
  singularities {I}: a comparison of counting functions.
\newblock arXiv:1309:2879, to appear in IMRN.

\bibitem{MR969545}
David~J. Wright.
\newblock Distribution of discriminants of abelian extensions.
\newblock {\em Proc. London Math. Soc. (3)}, 58(1):17--50, 1989.

\bibitem{densities}
Takehiko Yasuda.
\newblock Densities of rational points and number fields.
\newblock arXiv:1408.3912.

\bibitem{Yasuda:2013fk}
Takehiko Yasuda.
\newblock Toward motivic integration over wild {D}eligne-{M}umford stacks.
\newblock arXiv:1302.2982, to appear in the proceedings of ``Higher Dimensional
  Algebraic Geometry - in honour of Professor Yujiro Kawamata's sixtieth
  birthday".

\bibitem{wild-p-adic}
Takehiko Yasuda.
\newblock The wild {M}c{K}ay correspondence and $p$-adic measures.
\newblock arXiv:1412.5260.

\bibitem{Yasuda:2014fk2}
Takehiko Yasuda.
\newblock Wilder {M}c{K}ay correspondences.
\newblock arXiv:1404.3373, to appear in Nagoya Mathematical Journal.

\bibitem{MR2271984}
Takehiko Yasuda.
\newblock Motivic integration over {D}eligne-{M}umford stacks.
\newblock {\em Adv. Math.}, 207(2):707--761, 2006.

\bibitem{MR3230848}
Takehiko Yasuda.
\newblock The {$p$}-cyclic {M}c{K}ay correspondence via motivic integration.
\newblock {\em Compos. Math.}, 150(7):1125--1168, 2014.

\end{thebibliography}

\end{document}